\documentclass[reqno,8pt,draft]{amsart}

\usepackage[a4paper,left=35mm,right=35mm,top=30mm,bottom=30mm,marginpar=25mm]{geometry} 
\usepackage{amsmath}
\usepackage{esint}
\usepackage{amssymb}
\usepackage{amsthm}
\usepackage{amscd}
\usepackage[ansinew]{inputenc}
\usepackage{color}
\usepackage[final]{graphicx}
\usepackage{float}

\usepackage{rotating}
\usepackage{tikz}
\usetikzlibrary{shapes.geometric, arrows,positioning}
\tikzstyle{model} = [rectangle,  minimum height=2cm,font=\normalsize,text centered, draw=black, fill=black!5]
\tikzstyle{title}=[fill=black!20, text=black,rounded corners,font=\small\bfseries\sffamily]
\usepackage{tikz-cd}

\usepackage{xcolor}
\usepackage{todonotes}
\usepackage{changes}
\definechangesauthor[name={Ana}, color=magenta]{amr}
\definechangesauthor[name={Elvira}, color=blue]{ez}
\definechangesauthor[name={AnaNotes}, color=gray]{toerase}

\usepackage{lscape}

\renewcommand{\epsilon}{\varepsilon}

\numberwithin{equation}{section}

\newtheoremstyle{thmlemcorr}{10pt}{10pt}{\itshape}{}{\bfseries}{.}{10pt}{{\thmname{#1}\thmnumber{ #2}\thmnote{ (#3)}}}
\newtheoremstyle{thmlemcorr*}{10pt}{10pt}{\itshape}{}{\bfseries}{.}\newline{{\thmname{#1}\thmnumber{ #2}\thmnote{ (#3)}}}
\newtheoremstyle{defi}{10pt}{10pt}{\itshape}{}{\bfseries}{.}{10pt}{{\thmname{#1}\thmnumber{ #2}\thmnote{ (#3)}}}
\newtheoremstyle{remexample}{10pt}{10pt}{}{}{\bfseries}{.}{10pt}{{\thmname{#1}\thmnumber{ #2}\thmnote{ (#3)}}}
\newtheoremstyle{ass}{10pt}{10pt}{}{}{\bfseries}{.}{10pt}{{\thmname{#1}\thmnumber{ A#2}\thmnote{ (#3)}}}

\theoremstyle{thmlemcorr}
\newtheorem{theorem}{Theorem}
\numberwithin{theorem}{section}
\newtheorem{lemma}[theorem]{Lemma}
\newtheorem{corollary}[theorem]{Corollary}
\newtheorem{proposition}[theorem]{Proposition}

\theoremstyle{thmlemcorr*}
\newtheorem{theorem*}{Theorem}
\newtheorem{lemma*}[theorem]{Lemma}
\newtheorem{corollary*}[theorem]{Corollary}
\newtheorem{proposition*}[theorem]{Proposition}
\newtheorem{problem*}[theorem]{Problem}
\newtheorem{conjecture*}[theorem]{Conjecture}

\theoremstyle{defi}
\newtheorem{definition}[theorem]{Definition}

\theoremstyle{remexample}
\newtheorem{remark}[theorem]{Remark}
\newtheorem{example}[theorem]{Example}

\theoremstyle{ass}

\DeclareMathOperator{\esssup}{ess\,sup}

\newcommand{\R}{\mathbb{R}}

\def\XXint#1#2#3{{\setbox0=\hbox{$#1{#2#3}{\int}$} 
\vcenter{\hbox{$#2#3$}}\kern-.5\wd0}}

\definecolor{vg}{rgb}{0.0, 0.26, 0.15}

\title[Revisited convexity notions for $L^\infty$ variational problems]{Revisited convexity notions for $L^\infty$ variational problems}
\author{Ana Margarida Ribeiro}
\address{Center for Mathematics and Applications (NOVA Math) and Department of Mathematics, NOVA FCT, Quinta da Torre, 2829-516 Caparica, Portugal}
\email{amfr at fct.unl.pt}

\author{Elvira Zappale}
\address{Dipartimento di Scienze di Base ed Applicate per l'Ingegneria, Sapienza, Universit\'a di Roma, via Antonio Scarpa 16, 00161 Roma, Italy}
\email{elvira.zappale at uniroma1.it}

\begin{document}


\maketitle
\thispagestyle{empty}

\begin{abstract}
We address a deep study of the convexity notions that arise in the study of weak* lower semicontinuity of supremal functionals as well as those raised by the power-law approximation of such functionals. Our quest is motivated by the knowledge we have on the analogous integral functionals and aims at establishing a solid groundwork to ease any research in the $L^\infty$ context.
\vspace{8pt}


 
\vspace{8pt}

\noindent\textsc{Date:} \today.
\end{abstract}


\section{Introduction}

In these past decades there has been a growing interest towards $L^\infty$ variational problems because of their main applications. Indeed, in a first moment, they appeared empirically in the search for  bounds in optimal design problems such as determining the yield set of a polycrystal, or the first failure of a dielectric, in particular, in connection with power-law approximation. This latter method has shown to be a quite efficient procedure to describe the mentioned phenomena. In fact, not only it was adopted in physics literature (see \cite{HS, TW1,TW2, TW3}) but, later on,  a rigorous  mathematical justification was provided, see \cite{B, BHH, BM,  BMP, BN, BP,  BGP, CDP, EP, GNP, GK, KS, NTW,  PZ}.


Variational models in $L^\infty$ have also emerged in connection with Lipschitz extension problems \cite{BDM,J1} and as the energetic formulation of non linear partial differential equations (see for, by now, classical results \cite{A1}-\cite{A4},  and the more recent contributions, \cite{B1, BEJ}, also for higher order problems contained in \cite{AK, CKP} and the bibliography contained therein).

Then a wide literature has been developed, also in the non-Euclidean setting, starting from \cite{J1} and its quoting literature, in connection with Dirichlet forms, pre-fractal sets, Finsler structures, etc., see \cite{CCV, CV, GP, GXY, KSZ}, among a much wider scientific production. 

It is worth to recall that these mathematical models play also an important role in the context of optimal transport, game theory, partial differential equations, non-local problems also in connection with artificial intelligence problems, etc., see e.g. \cite{BBJ, BJR, BDP, BKDP, GZ, KZ, KRZ}.


Here, we are interested on those of the above models that are formulated through energy minimization. In particular, we consider what we call a {\sl supremal problem}, i.e. the  minimization of a functional 
\begin{equation}\label{intro supremal funcional}
F(u,O)=\operatorname*{ess\,sup}_{x\in \Omega} f(x, u(x), D u(x)),\ O\subseteq\Omega\text{ open },\ u\in W^{1,\infty}(O;\mathbb{R}^N)
\end{equation}
where $\Omega$ is an open and bounded set of $\mathbb{R}^n$ and $f:\Omega\times\mathbb{R}^n\times\mathbb{R}^{N\times n}\longrightarrow\mathbb{R}$ is a Carathéodory function. We call the function $f$ {\sl supremand}.

In contrast with the minimization of functionals in the integral form (that we assume the reader is familiar with, see \cite{Dbook,Fonseca-Leoni-unpub, Rindler} for an overview of the subject), the appropriate minimization in the present framework is the study of {\sl absolute minimizers}, see \cite[Definition 1.1]{CDP} and \cite{J}. This requires to deal with the minimization in a fixed domain, say, the minimization in the space $u_0+W_0^{1,\infty}(\Omega;\mathbb{R}^N)$ of the functional $F(\cdot,\Omega)$, for a given boundary condition $u_0$.

In the seminal paper by Barron, Jensen, and Wang \cite{BJW}, a necessary and sufficient condition on the supremand $f$ to the sequential weak* lower semicontinuity of $F(\cdot,\Omega)$ was found. This yields results on the existence of minimizers to the previous functional. The condition was named {\sl strong Morrey quasiconvexity} and it is recalled in Definition \ref{def convex notions}. To ease checking this condition in the applications, necessary and sufficient conditions to strong Morrey quasiconvexity were also introduced. Namely, 
the sufficient condition {\sl polyquasiconvexity} and the necessary conditions {\sl weak Morrey quasiconvexity} and {\sl rank one quasiconvexity}. 
 It is also worth to take into account that problem \eqref{intro supremal funcional} was already subject of interest in the scalar setting, i.e. $n=1$ or $N=1$, with the necessary and sufficient condition for the sequential weak* lower semicontinuity  of $F(\cdot,\Omega)$ on the supremand
$f$ detected by \cite{BJ1, BJL, B1} and known in the optimization literature as {\it quasiconvexity} and later on named {\it level convexity} by \cite{ABP, P0, P}, and used for the supremal representation by \cite{BCJ1, BCJ2,CP}.

In a rough manner, one has the following 
$$\begin{array}{c}f \text{ level convex }\Rightarrow\ f\text{ polyquasiconvex}\Rightarrow\ f\text{ strong Morrey quasiconvex}\Rightarrow\ \vspace{0.2cm}\\ 
\Rightarrow\ f\text{ weak Morrey quasiconvex}\Rightarrow\ f\text{ rank-one quasiconvex}.\end{array}$$
If one is acquainted with the Direct Method of the Calculus of Variations in the context of integral minimization and the related theory for vectorial problems, the previous chain of implications seems familiar and natural. As we will see, the specificity of supremal problems brings into play new features and, even the above implications shall be read with care under appropriate additional assumptions. Besides, the treatment of minimization problems through $L^p$ approximation, cf. \cite{ APSIMA, APESAIM, CDP, PZ}, brought into play other relevant conditions. 

In this paper we will focus on the properties of supremands enjoying the above convexity conditions. We trust the present paper provides a clear baseline to researchers dealing with problems in the field, as it gathers many properties and results disperse in the literature, accurately clarifying some features of the concepts under study, adding also novel insights contributing to a unified vision for the study of supremal problems. Moreover, this is a fundamental step to further proceed to our ultimate goal, that we postpone for a future work, that is to extend to the vectorial setting the previous work of the authors \cite{RZ}, namely, to provide conditions to ensure the existence of minimizers when the supremand $f$ lacks of the strong Morrey quasiconvexity and the Direct Method cannot be applied.

Next, we describe how the paper is organized as well as the ideas and questions that have driven our analysis. We note that, with the exception of Appendix B, in all our work we restrict to supremands only depending on the gradient variable. This allows to distinguish whether additional assumptions that one finds in the literature are intrinsic, or not, to the property under study.

Section 2 is devoted to the integral notion of {\sl quasiconvexity} (see Definition \ref{qcxdef}) which is the fundamental property associated to the sequential weak* lower semicontinuity in $W^{1,\infty}(\Omega;\mathbb{R}^N)$ of functionals of the form
$$I(u):=\int_\Omega f(Du(x))\,dx,\quad u\in u_0+W_0^{1,\infty}(\Omega;\mathbb{R}^N)$$ for a given function $f:\mathbb{R}^{N\times n}\to\mathbb{R}.$ Essentially, we recall existing results and, although the different context, this will be useful in the subsequent sections. 
	Aiming at formulating in the later sections our results in the greatest possible generality, we refer to the presentation of the monograph \cite{Fonseca-Leoni-unpub}. Besides, we highlight, that we establish, in Proposition \ref{new characterization of quasiconvexity}, a new characterization of quasiconvexity, motivated by the supremal notion of strong Morrey quasiconvexity treated later on, in Section 3. 

In Section 3 we consider in detail the notions introduced by Barron, Jensen and Young, previously mentioned. Several of the questions that we address in this section are motivated by properties that are well known in the integral setting. Being the first issue addressed the lower semicontinuity encoded, or not, by the notions under our attention. As we will see in Proposition \ref{fSMqcxOmegaflsc}, among the notions considered in this section, only the strong Morrey quasiconvexity encodes this property. 

Another question that leaded our investigation was, whether in the case of the supremand $f$ is strong Morrey quasiconvex and the boundary condition $u_0$ is an affine map, $u_0$ is also a minimizer for the functional $F(u,\Omega)$ in \eqref{intro supremal funcional}. The analogue to this in the integral setting is well known and amounts to the fact that the quasiconvexity notion is independent of the domain where the integral is considered. Therefore, we are led to the question of the independence of domain of strong Morrey quasiconvex functions. Besides, if we stick to the starting point of this discussion, our question is precisely equivalent to the independence of domain of weak Morrey quasiconvex functions. While for weak Morrey quasiconvexity we obtained a positive answer, cf. Proposition \ref{lemma weak independent}, and thus, we get that affine boundary conditions are minimizers to the problem described above, cf. Corollary \ref{Corollary 3.8}, the independence of domain for strong Morrey quasiconvexity was only ensured under some conditions on the sets, in particular, its convexity, see Proposition \ref{proposition independence of domain}. We note that the result on independence of the domain in the definition of strong Morrey quasiconvexity has been obtained by exploiting the lower semicontinuity of the related supremal functional, requiring to adapt results in \cite{BJW}. This is left to Appendix A. We just observe here that the independence of the domain in the convexity notions combines well with the fact that, in the nonhomogeneous setting, the supremal representation in terms of suitably 'convex' densities requires weakly* lower semicontinuity in every domain \cite[see counterexamples]{CP, P00, P0} and the bibliography contained therein.  

A deep understanding of minimization of integral functionals shows that the condition which is intrinsic to the weak* lower semicontinuity is the equivalent condition to quasiconvexity which is given by \eqref{qcxdef}, but testing on periodic functions. This is another source of research: whether in the notion of strong Morrey quasiconvexity, periodic functions can be considered. In this point, our analysis is not conclusive, motivating us to introduce the concept of {\sl periodic-weak Morrey quasiconvexity}. 

Still in Section 3, we investigate how do the notions of convexity introduced in this section relate between them. Our aim is to obtain an exhaustive study of these relations, therefore, whenever possible, we also provide counter-examples and we end the section with a list of the relations for which a satisfactory answer was not obtained. Also a proof of a characterization in terms of supremal Jensen's inequality involving probability measures under very mild assumptions is given in Appendix C.

In Section 4, our interest is to relate strong Morrey quasiconvexity with the convexity concepts raised by power-law approximation, not only as a way to deal with lower semicontinuity of $L^\infty$-variational problems, but also in order to rigorously obtain the latter ones by means of variational convergence departing from $L^p$- type norm functionals. More precisely, we relate strong Morrey quasiconvexity with the notions of {\sl ${\rm curl}_{(p>1)}$-Young quasiconvexity, ${\rm curl}$-Young quasiconvexity,} and {\sl ${\rm curl}-\infty$ quasiconvexity}, cf. Definition \ref{def more convex notions}, not necessarily under these names in the literature. As in Section 3, we also provide some counterexamples and we list some open questions of interest. In particular, we will see that coercivity always plays a crucial role. To this end, we start recalling the counterexamples to representation of weakly* lower semicontinuous supremal functionals, depending on gradients, in terms of non-homogeneous level convex densities (of the form $f(x,\xi)$) already in the scalar case, see \cite{P0} and the bibliography contained therein. Thus it arises naturally the question of comparing the notions providing sufficient conditions for representation of lower semicontinuous supremal functionals under coercivity hypotheses.  
Finally, due to the deep connections with Young measures, appearing already in some definitions, we will provide in Appendix B a new proof of sufficiency of $\rm curl$-Young quasiconvexity for weak* lower semicontinuity of supremal functionals (also in the nonhomogeneous setting). 

For the sake of completeness, we consider, in Section 5, the interplay between the convexity notions arising in the integral and the supremal settings. 

We will leave to further studies the comparison with supremal convexity notions using the duality theory in Convex analysis as in \cite{BL} and \cite{P00} or making use of the intrinsic distances as, e.g., in \cite{GPP} and \cite{GP}, or rephrasing the notions exploiting 
the connections with variational unbounded integral functionals and/or differential inclusions, see \cite{KZ, PZ, Z}.

\subsection{Notation} 
In the sequel we will make use of the following notation
\begin{itemize}
	\item We denote by $Q$ the unit cube of $\mathbb R^n$ centered at the origin with side length $1$, i.e. $Q:= \left(-\frac{1}{2}, \frac{1}{2}\right)^n$. 
	\item By $\mathcal L^n$ we denote the $n$-dimensional Lebesgue measure.
	\item For any set $E\subset \mathbb R^d$, $\chi_E$ denotes its characteristic function, i.e. $\chi_E(x)=\left\{\begin{array}{ll}
	1 &\hbox{ if } x \in E,\\
0 & \hbox{otherwise.}\end{array}\right.$
\item For every $\Omega\subset \mathbb R^n$ bounded open set with Lipschitz boundary, by $W^{1,\infty}_0(\Omega;\mathbb R^N)$ we denote the set $W^{1,\infty}(\Omega;\mathbb R^N)\cap W^{1,1}_0(\Omega;\mathbb R^N)$, where the latter set is the $W^{1,1}$-closure of $C^\infty_c(\Omega;\mathbb R^N)$.
\item For any cube $C\subset \mathbb R^n$, by $W^{1,\infty}_{\rm per}(C;\mathbb R^N)$ we denote the subset of $W^{1,\infty}(\mathbb R^n;\mathbb R^N)$,  made by $C$-periodic functions.
\end{itemize}


\section{A review of the integral notion of quasiconvexity and of some of its properties}

We recall the definition of quasiconvex function, fundamental in the minimization of vectorial integral functionals. A classic reference on this subject is the monograph \cite{Dbook}. In the sequel we sometimes refer to \cite{Fonseca-Leoni-unpub} where the quasiconvexity notion is given without requiring {\sl a priori} the local boundedness of the function. This shall be useful below when dealing with curl-$\infty$ quasiconvex functions. We call the attention for the new characterization of quasiconvexity established in Proposition \ref{new characterization of quasiconvexity}.

\begin{definition}\label{defqcx}
\label{qcx}A Borel measurable function $g:\mathbb R^{N \times n} \to \mathbb R$ is said quasiconvex if 
\begin{equation}\label{qcxdef}
g(\xi)\leq\int_Q g(\xi +D \varphi(x))\,dx
\end{equation}
for every $\xi\in\mathbb{R}^{N\times n}$ and for every $\varphi \in W^{1,\infty}_0(Q;\mathbb R^N)$, where $Q:= \left(-\frac{1}{2}, \frac{1}{2}\right)^n$. 
\end{definition}

\begin{remark}\label{rem-qcxdef}
\begin{enumerate}
\item[(i)] In the definition of quasiconvexity, one can also consider functions taking values in $[-\infty,\infty]$ but many properties may fail in this case. (See \cite{Fonseca-Leoni-unpub} and \cite{BM}.)

\item [(ii)] The notion of quasiconvexity is given in \cite{Fonseca-Leoni-unpub} by testing on the class of Lipschitz functions, 

$$\mathrm{Lip}_0(O;\mathbb{R}^N):=\left\{\varphi:O\longrightarrow\mathbb{R}^N|\ \varphi \text{ is Lipschitz in }\overline{O}\text{ and }\varphi=0 \text{ on the boundary}\right\},$$
which in the particular case of the set $O= Q$, coincides with ours.

\item[(iii)] As shown in \cite{Fonseca-Leoni-unpub}, 
 if $g$ is real valued then it is locally-Lipschitz, hence continuous and locally bounded.

\item[(iv)] In \eqref{qcxdef} 
 the cube $Q:=\left(-\frac{1}{2},\frac{1}{2}\right)^n$ can be replaced by any bounded open set $\Omega$ (averaging the integral in \eqref{qcxdef} by the measure of $\Omega$), {\sl cf.} \cite[Proposition 5.11]{Dbook}  (see also \cite{Fonseca-Leoni-unpub} 
). 
If, moreover, $\Omega$ has regular boundary, then the set of functions $\mathrm{Lip}_0(\Omega;\mathbb{R}^N)$ coincides with $W^{1,\infty}_0(\Omega;\mathbb{R}^N)$.

\item[(v)] Arguing as in 
\cite{Fonseca-Leoni-unpub}, if $g$ is real valued (then it is locally bounded) one can replace in the definition of quasiconvexity, via reverse Fatou's lemma, $W^{1,\infty}_0$ by $C_c^\infty$. 

\item[(vi)] We can replace $W^{1,\infty}_0(Q;\mathbb R^N)$ by $W^{1,\infty}_{\mathrm per}(Q;\mathbb R^N)$ as well as, {\sl cf. \cite{Fonseca-Leoni-unpub}}, by 

\noindent $\left\{\varphi\in W^{1,\infty}_{\rm loc}(\mathbb R^n;\mathbb R^N)|\ 
D\varphi\text{ is } Q\hbox{-periodic and }\int_Q D\varphi(x)\,dx=0\right\}$. Moreover, the unit cube $Q$ can be replaced by any cube, provided that the integral in \eqref{qcxdef} is averaged by the measure of the cube. 

\item[(vii)] The notion of quasiconvexity coincides with the $\mathcal{A}$-quasiconvexity ({\sl cf.} \cite[Remark 3.3]{FMAq}) in the case $\mathcal{A}=\mathrm{curl}$. 
\end{enumerate}
\end{remark}

Next we provide a new characterization of quasiconvexity, stemmed by the results in \cite{BJW}, in particular Proposition 2.4. For the readers convenience we present the proof.

\begin{proposition}\label{new characterization of quasiconvexity}
Let $g:\mathbb{R}^{N\times n}\longrightarrow\mathbb{R}$ be a Borel measurable function. Consider the following condition 
$$\forall\ \varepsilon>0\ \forall\ \xi\in \mathbb{R}^{N\times n}\ \forall\ K>0\ \exists\ \delta=\delta(\varepsilon, K,\xi)>0:$$
\begin{equation}\label{AltQ}
\left.\begin{array}{l}\varphi\in W^{1,\infty}(Q;\mathbb{R}^N)\vspace{0.2cm}\\ ||D\varphi||_{L^\infty(Q)}\le K\vspace{0.2cm}\\ \max_{x\in\partial Q}|\varphi(x)|\le \delta\end{array}\right\}\Longrightarrow g(\xi)\le \int_Q g(\xi+D\varphi(x))\,dx+\varepsilon.
\end{equation}
One has that $g$ is quasiconvex if and only if $g$ satisfies \eqref{AltQ}.
\end{proposition}

\begin{proof} Quasiconvexity follows immediately from \eqref{AltQ} applied to test functions $\varphi\in W^{1,\infty}_0(Q;\mathbb{R}^N)$, letting $\varepsilon\to 0$. 

To prove the reverse implication let $\varepsilon>0$, $\xi\in\mathbb{R}^{N\times n}$ and $K>0$ be arbitrary. Let $\varphi\in W^{1,\infty}(Q;\mathbb{R}^N)$ be such that $\|D \varphi\|_{L^\infty(Q)}\le K$ and $\max_{x \in \partial Q}|\varphi(x)|\le \delta$ for some $\delta$ to be chosen later. Then, having in mind that, by the Lipschitz continuity of $\varphi$ in $Q$, $|\varphi(y)-\varphi(x)|\le K|y-x|$ for any $x\in\partial Q$ and $y\in Q$, one has $|\varphi (y)| \leq 2 \delta$ for any $y =(y_1,\dots, y_n) \in Q$, with $|y_i | \geq 1/2 - \eta$ for some $1 \leq i \leq n$, provided that $\eta\leq\frac{\delta}{K}$.

Now, given $\eta>0$ small, let $\psi_\eta \in C^1(\mathbb R^n)$ be such that $0\le \psi_\eta(y)\le 1$ for all $y\in \mathbb R^n$, 
$\psi_\eta \equiv 1$ for $y \in (1 -\eta)Q$, $\psi_\eta \equiv 0$ for $y \not \in Q$, and $|D\psi_\eta(y)|< \frac{c_0}{\eta}$ for every $y\in \mathbb{R}^n$ for some constant $c_0$ independent of $\eta$.

The quasiconvexity of $g$ entails that
\begin{align*}
g(\xi) & \leq \int_Q g(\xi +D(\psi_\eta \varphi)(x))\,dx\\
& = \int_Q g(\xi +D \varphi(x))\,dx + \int_{Q\setminus(1-\eta)Q} (g(\xi + D(\psi_\eta \varphi)(x))-g(\xi+D\varphi(x)))\,dx.
\end{align*}

To conclude the proof it suffices to estimate the latter integral on the right hand side by $\varepsilon$. To achieve this, observe that, since $g$ is quasiconvex, by (iii) in Remark \ref{rem-qcxdef}, there exists a constant $L$ depending on $\xi$ and $K$ above such that
\begin{align}\label{L}
|g(\zeta)- g(\zeta')|\leq L |\zeta-\zeta'| \hbox{ for every } \zeta, \zeta' \in B_R(\xi),
\end{align}
where the latter denotes the ball in $\mathbb R^{N\times n}$ centered at $\xi$ with radius $R:=K(1+4\,c_0)$.

Observe that $\|D \varphi\|_{L^\infty(Q)}\leq K\le R$ and, if we take $\eta=\frac{\delta}{2\, K}$, then $\|D(\psi_\eta \varphi)\|_{L^\infty(Q)}\le  R$. Therefore, by \eqref{L}, and observing that $\mathcal{L}^n(Q\setminus(1-\eta)Q)\le c_1\eta$, for some constant $c_1$ only depending on the dimension $n$, one gets
\begin{align*}
&\displaystyle{\int_{Q\setminus(1-\eta)Q}\left| g(\xi + D(\psi_\eta \varphi)(x))-g(\xi + D \varphi(x)) \right|\,dx}\leq \\
&\leq L \int_{Q\setminus (1-\eta)Q} (|D \varphi(x)| + |D \psi_\eta(x)|\,|\varphi(x)|)\,dx\\
&\leq L (K c_1\eta + 2c_0c_1 \delta)= c_1L( \tfrac{1}{2} + 2c_0) \delta.
\end{align*}
Then it is enough to take $\delta$ sufficiently small so that the last term is smaller than $\varepsilon$.
\end{proof}

\begin{definition}\label{qcxenv}
Let $g:\mathbb R^{N \times n} \to \mathbb R$ be a Borel measurable function. The greatest quasiconvex function below $g$ is called the quasiconvex envelope of $g$ and it is denoted by  $\mathcal Q g$, i.e. $\mathcal Q g:\mathbb R^{N\times n}\to \mathbb{R}$ is such that
$$\mathcal Q g(\xi):=\sup\left\{h(\xi)|\ h:\mathbb R^{N \times n} \to \mathbb R,\ h(\xi) \leq g(\xi) \text{ for all }\xi\in\mathbb R^{N \times n},\ h \hbox{ is quasiconvex}\right\}.$$
\end{definition}

\begin{remark}\label{remqcx}
Note that, for $g$ taking values in $\mathbb{R}$, $\mathcal Q g$ is well defined and it is a Borel measurable function, cf. 
	\cite{Fonseca-Leoni-unpub}.
\end{remark}

The following lemma will be useful for the sequel of this paper. It relies on results contained in \cite{Fonseca-Leoni-unpub}.

\begin{lemma}\label{lemmaenvelope} Let $g:\mathbb{R}^{N\times n}\longrightarrow\mathbb{R}$ be a Borel measurable function. Then, for every $\xi\in\mathbb{R}^{N\times n}$, one has
\begin{align*}
	\mathcal Q g(\xi)&= \inf\left\{\int_Q g(\xi +D \varphi(x))\,dx: \varphi \in W^{1,\infty}_{\rm loc}(\mathbb R^n;\mathbb R^N),\hbox{ with }D\varphi\ Q\hbox{-periodic and }\int_QD \varphi(x)\,dx=0\right\}\\
	&= \inf\left\{\int_Q g(\xi +D \varphi(x))\,dx: \varphi \in W^{1,\infty}_{\rm per}( Q;\mathbb R^N)\right\}\\
	&= \inf\left\{\int_Q  g(\xi +D \varphi(x))\,dx: \varphi \in W^{1,\infty}_0( Q;\mathbb R^N)\right\} \\
	&= \inf\left\{\frac{1}{\mathcal L^n(\Omega)}\int_\Omega  g(\xi +D \varphi(x))\,dx: \varphi \in W^{1,\infty}_0(\Omega;\mathbb R^N)\right\}.
\end{align*}
if $\Omega\subseteq\mathbb{R}^n$ is open and bounded with $\mathcal L^n(\partial \Omega)=0$.
\end{lemma}

\begin{proof}
According to  \cite{Fonseca-Leoni-unpub}
the quasiconvex envelope is {\it strongly quasiconvex},  
thus
\begin{align*}
	\mathcal Q g(\xi)&\leq \inf\left\{\int_Q \mathcal Qg(\xi +D \varphi(x))\,dx: \varphi \in W^{1,\infty}_{\rm loc}(\mathbb R^n;\mathbb R^N),\hbox{ with }D\varphi\ Q\hbox{-periodic and }\int_QD \varphi(x)\,dx=0\right\}\\
	&\leq \inf\left\{\int_Q g(\xi +D \varphi(x))\,dx: \varphi \in W^{1,\infty}_{\rm loc}(\mathbb R^n;\mathbb R^N),\hbox{ with }D\varphi\ Q\hbox{-periodic and }\int_QD \varphi(x)\,dx=0\right\}\\
	&\leq \inf\left\{\int_Q g(\xi +D \varphi(x))\,dx: \varphi \in W^{1,\infty}_{\rm per}( Q;\mathbb R^N)\right\}\\
	&\leq \inf\left\{\int_Q  g(\xi +D \varphi(x))\,dx: \varphi \in W^{1,\infty}_0( Q;\mathbb R^N)\right\} \\
	&= \inf\left\{\int_Q  g(\xi +D \varphi(x))\,dx: \varphi \in Lip_0( Q;\mathbb R^N)\right\} \\
	&\leq \mathcal Q g(\xi),
\end{align*}
where in the equality we used Remark \ref{rem-qcxdef} (ii) and in the last inequality we used a result in 
\cite{Fonseca-Leoni-unpub}.
Thus all the formulas above coincide. Moreover,  in view of Remark \ref{rem-qcxdef} (iv) exploiting similar arguments as in the chain of inequalities above, and the invariance of the domain, proven in 
{Fonseca-Leoni-unpub} $Q g(\xi)$ it also coincides with 
$$\inf\left\{\int_\Omega  g(\xi +D \varphi(x))\,dx: \varphi \in W^{1,\infty}_0(\Omega;\mathbb R^N)\right\}$$ and with 
$$\inf\left\{\int_\Omega  g(\xi +D \varphi(x))\,dx: \varphi \in Lip_0( \Omega;\mathbb R^N)\right\}$$ provided $\Omega\subseteq\mathbb{R}^n$ is open and bounded with $\mathcal L^n(\partial \Omega)=0$.

\end{proof}

The following definition has been introduced by Ball and Murat (cf. \cite[Definition 2.1]{BM}). 

\begin{definition}\label{W1pqcx} Let $1\leq p \leq \infty$.
A Borel  function $g: \mathbb R^{N \times n}\to \mathbb R$ is $W^{1,p}$-quasiconvex at $\xi \in \mathbb R^{N \times n}$ if
\begin{equation}\label{w1pqcxeq}
g(\xi) \leq \int_Q g(\xi + D \varphi(x))\,dx
\end{equation}
for every $\varphi \in W^{1,p}_0(Q;\mathbb R^{N})$.
We say that $g$ is $W^{1,p}$-quasiconvex if it is $W^{1,p}$-quasiconvex
at every $\xi \in \mathbb R^{N\times n}$.
\end{definition}

\begin{remark}\label{W1pqcx}
\begin{enumerate}
\item[(i)] Note that $W^{1,\infty}$-quasiconvexity is the quasiconvexity introduced in Definition \ref{defqcx}.
\item[(ii)] The above definition can be given also when the range of $g$ is $[-\infty,+\infty]$. 
\item[(iii)]The set $Q$ can be replaced by any bounded open set $\Omega$ such that $\mathcal L^n(\partial \Omega)=0$ replacing the integral by an averaged integral in $\Omega$.
\item[(iv)]	If $g$  is $W^{1,p}$-quasiconvex
for some $1 \leq p \leq+\infty$ , then it is $W^{1,q}$-quasiconvex for all $p \leq q \leq +\infty$ (cf. \cite[Remark 2.2]{BM}). Thus
quasiconvexity and $W^{1,1}$-quasiconvexity are, respectively, the weakest and the
strongest condition.
\item[(v)] If $g : \mathbb R^{N\times n} \to \mathbb{R}$ 
and satisfies the growth condition
\begin{equation}\label{gpgab}
	g(\eta)\leq C(1+ |\eta|^p)
	\end{equation}
 for all $\eta \in \mathbb R^{N\times n}$, for some $C>0$, and $1\leq p<+\infty$, then $g$ is $W^{1,p}$-quasiconvex
at $\xi$ if and only if for every bounded and open set $\Omega$,  
$$g(\xi)\leq \frac{1}{\mathcal L^{n}(\Omega)}\int_\Omega g(\xi + D \varphi(x))\,dx$$
for every $\varphi \in C^\infty_c(\Omega;\mathbb R^{N})$. 
\item[(vi)] Also, 
if $g$ is $W^{1,p}$-quasiconvex and satisfies the coercivity condition
\begin{equation}\label{gpgbe}
g(\eta) \geq C'(|\eta|^p -1)
\end{equation}
for all $\eta \in \mathbb R^{N\times n}$, for some constant $C' > 0$, and $1 < p < +\infty$ , then $g$
is $W^{1,1}$-quasiconvex.
\end{enumerate}
\end{remark}

The previous remark allows to state, in the spirit of the characterizations obtained in Lemma \ref{lemmaenvelope} for the quasiconvex envelope of a function, an alternative formula in terms of $W^{1,p}_0(\Omega;\mathbb R^N)$ test functions. 
\begin{lemma}\label{lemmaenvelope2} Let $1 \leq p < +\infty,$ and let $g:\mathbb{R}^{N\times n}\longrightarrow\mathbb{R}$ be a Borel measurable function satisfying the growth condition \eqref{gpgab}. Let $\Omega$ be open, bounded with $\mathcal{L}^n(\partial\Omega)=0$. Then 
\begin{align*}
\mathcal Q g(\xi) = \inf\left\{\frac{1}{\mathcal L^n(\Omega)}\int_\Omega g(\xi +D \varphi(x))\,dx: \varphi \in W^{1,p}_0( \Omega;\mathbb R^N)\right\}.
\end{align*}
Furthermore, under the extra assumption that $g$ is upper semicontinuous, one has
\begin{align*}
\mathcal Q g(\xi) = \inf\left\{\int_Q g(\xi +D \varphi(x))\,dx: \varphi \in W^{1,p}_{\rm per}(Q;\mathbb R^N)\right\}.
\end{align*}
\end{lemma}

\begin{proof}
The first statement follows from the previous remark, observing that under our assumptions $\mathcal Q g$ is Borel measurable, quasiconvex and satisfies the same growth condition as $g$. Indeed, by Remark \ref{W1pqcx} (v) and then by (iii) one has 
\begin{align}\label{charactW1p}
	\mathcal Q g(\xi)&\leq \inf\left\{\frac{1}{\mathcal{L}^n(\Omega)}\int_\Omega \mathcal Qg(\xi +D \varphi(x))\,dx: \varphi \in W^{1,p}_0(\Omega;\mathbb R^N)\right\}\\
	&\leq \inf\left\{\frac{1}{\mathcal{L}^n(\Omega)}\int_\Omega g(\xi +D \varphi(x))\,dx: \varphi \in W^{1,p}_0( \Omega;\mathbb R^N)\right\}\nonumber\\
	&\leq \inf\left\{\frac{1}{\mathcal{L}^n(\Omega)}\int_\Omega g(\xi +D \varphi(x))\,dx: \varphi \in W^{1,\infty}_0( \Omega;\mathbb R^N)\right\}\nonumber\\
	&= \mathcal Qg(\xi)\nonumber
\end{align}
invoking the characterization of $\mathcal Q g$ provided in Lemma \ref{lemmaenvelope} to obtain the last identity.

\noindent For the second statement, we start by observing that one inequality follows from the fact that $W^{1,\infty}_{\rm per}(Q;\mathbb{R}^N)$ $\subseteq$ $ W^{1,p}_{\rm per}(Q;\mathbb{R}^N)$ and by the characterization of $\mathcal Q g$ provided by Lemma \ref{lemmaenvelope}. For the other inequality,
consider for a fixed function $\varphi \in W^{1,p}_{\rm per}(Q;\mathbb R^N)$, the convolution with a sequence of mollifiers $(\rho_\varepsilon)_{\varepsilon}$, defined as $
		\rho_\varepsilon(x):= \frac1{\varepsilon ^n}\rho\Big(\frac x\varepsilon\Big), $
	where
	\begin{align*}
		\rho \in C_c^\infty(\mathbb R^n), \qquad {\rm supp} \rho \subset B_1,\qquad \rho\geq0,\qquad \rho(-x)=\varrho(x),
		\qquad \text{and}
		\quad \int_{B_1} \rho(x)\,d x=1,
		\end{align*}
denoting $B_1$ the unit ball in $\mathbb R^n$ centered at the origin. Note that $\varphi\ast\rho_\varepsilon$ is still a periodic function and moreover it belongs to $W^{1,\infty}(\mathbb{R}^n;\mathbb{R}^N)$. Then by the quasiconvexity of $\mathcal Q g$, one has 
$$\mathcal Q g(\xi)\le \limsup_{\varepsilon\to 0} \int_Q g(\xi+D(\varphi\ast\rho_\varepsilon)(x))\,dx.$$ Note that, by \eqref{gpgab}, the right-hand side in the previous inequality is bounded. This allows to use the reversed Fatou's lemma and get 
$$\mathcal Q g(\xi)\le \int_Q \limsup_{\varepsilon\to 0} g(\xi+D(\varphi\ast\rho_\varepsilon)(x))\,dx.$$ Finally, the upper semicontinuity assumption ensures that the integrand converges pointwise to $g(\xi+D\varphi(x))$, as desired.
\end{proof}


\section{On some convexity notions for functions and sequential weak* l.s.c. of $L^\infty$ functionals}

In this section we revisit some convexity notions previously introduced by Barron, Jensen and Wang \cite{BJW} in the context of $L^\infty$ functionals. These notions are related to the problem of existence of minimizers for supremal functionals and to the not fully understood notion of quasiconvexity for unbounded integral functionals (see \cite[Lemma 1.4]{BJW} and the last section in \cite{HK}). Our goal, in this section, is to better understand each of these notions as well as the relations between them.  The questions addressed here are motivated by the knowledge on the analogous notions in the context of integral minimization problems.

Once the convexity notions are introduced, we consider, in a first moment, some properties that are intrinsic to them. Namely, lower semicontinuity and invariance on the domain. We also get a characterization of periodic-weak Morrey quasiconvexity and existence of minimizers for some class of supremal problems. Afterwards, in Subsection \ref{subsection Hierarchies}, trying to be exhaustive, we exploit which are the conditions that are sufficient and which are necessary.

\begin{definition}\label{def convex notions}
Let $N, n \in \mathbb N$ and let $f:\mathbb{R}^{N\times n}\longrightarrow \mathbb{R}$.

\begin{enumerate}
\item The function $f$ is said to be  level convex if, for every $\xi,\eta\in\mathbb{R}^{N\times n}$ and for every $0<\lambda<1$, one has $$f(\lambda\xi+(1-\lambda)\eta)\le \max\{f(\xi),f(\eta)\}.$$
\item The function $f$ is said to be  polyquasiconvex if, there exists a level convex function $g:\mathbb{R}^{\tau(n,N)} $ $\longrightarrow \mathbb{R}$ such that, for every $\xi\in\mathbb{R}^{N\times n}$, $$f(\xi)=g(T(\xi)),$$ where
$$\tau(n,N):=\sum_{s=1}^{\min\{n,N\}}\sigma(s),\quad	\text{with}\quad \sigma(s)=\binom{N}{s}\binom{n}{s}=\frac{N!\,n!}{(s!)^2(N-s)!(n-s)!}$$ and $T(\xi)$ is a vector with all the minors of $\xi$, namely
$$T(\xi):=(\xi,\mathrm{adj}_2\xi,\dots,\mathrm{adj}_{\min\{n,N\}}\xi)$$ being $\mathrm{adj}_s\xi$ $(2\le s\le \min\{n,N\})$ the matrix of all $s\times s$ minors of $\xi$. 
\item \label{def strong}Assume that the function $f$ is Borel measurable. We say that $f$ is  strong Morrey quasiconvex if 
$$\forall\ \varepsilon>0\ \forall\ \xi\in \mathbb{R}^{N\times n}\ \forall\ K>0\ \exists\ \delta=\delta(\varepsilon, K,\xi)>0:$$
$$\left.\begin{array}{l}\varphi\in W^{1,\infty}(Q;\mathbb{R}^N)\vspace{0.2cm}\\ ||D\varphi||_{L^\infty(Q;\mathbb{R}^{N\times n})}\le K\vspace{0.2cm}\\ \max_{x\in\partial Q}|\varphi(x)|\le \delta\end{array}\right\}\Longrightarrow f(\xi)\le \operatorname*{ess\,sup}_{x\in Q} f(\xi+D\varphi(x))+\varepsilon.$$
\item\label{def weak} Assume that the function $f$ is Borel measurable. We say that $f$ is weak Morrey quasiconvex if $$\displaystyle f(\xi)\le \operatorname*{ess\,sup}_{x\in Q}f\left( \xi+ D\varphi\left( x\right)  \right),\ \forall\ \xi\in\mathbb{R}^{N\times n},\ \forall\ \varphi\in W_{0}^{1,\infty}(Q;\mathbb{R}^N).$$
\item \label{def periodic-weak} Assume that the function $f$ is Borel measurable. We say that $f$ is  periodic-weak Morrey quasiconvex if, for every $\xi\in\mathbb{R}^{N\times n}$ and for every $\varphi\in W_{\rm{per}}^{1,\infty}(Q;\mathbb{R}^N)$, $$\displaystyle f(\xi)\le \operatorname*{ess\,sup}_{x\in Q}f\left( \xi+ D\varphi\left( x\right)  \right).$$
\item The function $f$ is said to be rank-one quasiconvex if, for every $\xi,\eta\in\mathbb{R}^{N\times n}$ such that $\mathrm{rank}(\xi-\eta)=1$ and for every $0<\lambda<1$, one has $$f(\lambda\xi+(1-\lambda)\eta)\le \max\{f(\xi),f(\eta)\}.$$
\end{enumerate}
\end{definition}

\begin{remark}\label{remark periodic-weak} 
	\begin{itemize}
\item[(i)]Regarding the notion of strong Morrey quasiconvexity, to ease the parallel with the integral setting, one shall have in mind the characterization of quasiconvexity provided in Proposition \ref{new characterization of quasiconvexity}. In this way, we obtain that the supremal version of \eqref{AltQ} provides the notion of strong Morrey quasiconvexity, while the supremal version of \eqref{qcxdef} leads to weak Morrey quasiconvexity. As we shall see in Proposition \ref{proposition-counter-implications} these two notions, weak and strong Morrey quasiconvexity, do not coincide.  

\item[(ii)]We will work with the conditions defining strong Morrey quasiconvexity and weak Morrey quasiconvexity on domains other than the cube $Q$, namely $\Omega\subseteq\mathbb{R}^n$. In that case we will refer to those conditions as strong Morrey quasiconvexity in $\Omega$ or weak Morrey quasiconvexity in $\Omega$. In Propositions \ref{lemma weak independent} and \ref{proposition independence of domain} we will see these notions are independent of the domain in some appropriate classes of sets. In an analogous way, we will also refer to periodic-weak Morrey quasiconvexity in a cube $C\subseteq\mathbb{R}^n$ if the inequality in \eqref{def periodic-weak} is valid for test functions in $W^{1,\infty}_{\rm per}(C;\mathbb{R}^N).$

\item[(iii)]The definition of periodic-weak Morrey quasiconvexity is new as a definition, but it was already used in \cite[Lemma 2.8]{BJW} through a formulation, that we will prove to be equivalent in Proposition \ref{Prop equiv periodic-weak}. It appears as an intermediate step to prove that sequential weak* lower semicontinuity of a supremal functional of the form $F(\cdot,\Omega)$, as in \eqref{intro supremal funcional},  implies weak Morrey quasiconvexity. Example \ref{exampleC} below provides a counter-example to the reverse implication. The notion of periodic-weak Morrey quasiconvexity also coincides with the notion of $\mathcal{A}$-weak quasiconvexity considered in \cite{APSIMA} in the case of the ${\rm curl}$ operator, since by Proposition \ref{Prop equiv periodic-weak}, periodic-weak Morrey quasiconvexity can be tested on functions with periodic gradients.
\end{itemize}
\end{remark}

Having in mind the relevant convexity notions to treat minimization problems of integral form (see \cite{Dbook}), one should note that since convex functions are continuous in their effective domain, the lower semicontinuity is encoded in polyconvexity, quasiconvexity and rank-one convexity. This is not the case in the context under our attention here. Indeed, there are level convex, polyquasiconvex, weak and periodic-weak Morrey quasiconvex, and rank-one quasiconvex functions that are not lower semicontinuous. On the other hand, it will be seen, cf. Proposition \ref{fSMqcxOmegaflsc}, that strong Morrey quasiconvex functions are always lower semi-continuous. To this end, we start by recalling the following  preparatory result (see \cite[Proposition 2.5]{BJW}), which proof is presented for the convenience of the reader.

\begin{lemma}\label{2.5inOmegabddopen0}
	Let $f:\mathbb{R}^{N\times n}\longrightarrow \mathbb{R}$ be a Borel measurable function that is strong Morrey quasiconvex on a bounded and open set $\Omega\subseteq\mathbb{R}^n$ with Lipschitz boundary {\sl i.e.}
	\begin{equation}\label{SMQcxOmega}\forall\ \varepsilon>0\ \forall\ \xi\in \mathbb{R}^{N\times n}\ \forall\ K>0\ \exists\ \delta=\delta(\varepsilon, K,\xi)>0:$$
		$$\left.\begin{array}{l}\varphi\in W^{1,\infty}(\Omega;\mathbb{R}^N)\vspace{0.2cm}\\ ||D\varphi||_{L^\infty(\Omega;\mathbb{R}^{N\times n})}\le K\vspace{0.2cm}\\ \max_{x\in\partial \Omega}|\varphi(x)|\le \delta\end{array}\right\}\Longrightarrow f(\xi)\le \operatorname*{ess\,sup}_{x\in \Omega} f(\xi+D\varphi(x))+\varepsilon.
	\end{equation}
	Then, for every $\xi \in \mathbb R^{N\times n}$ and for every sequence $(\varphi_k)_{k\in\mathbb{N}}\subseteq W^{1,\infty}(\Omega;\mathbb R^N)$ weakly* converging to $0$ in $W^{1,\infty}(\Omega;\mathbb R^N)$, it results
	$$f(\xi)\leq \liminf_{k \to\infty}\operatorname*{ess\,sup}_{x\in \Omega} f(\xi + D \varphi_k(x)).$$
\end{lemma}

\begin{proof}[Proof]
	The 
	weak* convergence of $\varphi_k$ to $0$ in $W^{1,\infty}(\Omega;\mathbb R^N)$ entails that $(\varphi_k)_{k\in\mathbb{N}}$ is bounded in $W^{1,\infty}(\Omega;\mathbb R^N)$. In particular, there exists $K>0$ such that $\|D\varphi_k\|_{L^\infty(\Omega;\mathbb{R}^{N\times n})}\leq K$, for every $k \in \mathbb N$. Moreover, by Rellich-Kondrachov theorem, $\varphi_k\in C(\overline{\Omega};\mathbb{R}^N)$ and $\varphi_k \to 0$ strongly in $L^\infty(\Omega;\mathbb R^N)$. Therefore, given $\xi \in \mathbb R^{N\times n}$, $\varepsilon>0$, and $\delta=\delta(\varepsilon, K,\xi)>0$ as in the assumption, one has, for sufficiently large $k$, $\sup_{x \in \partial \Omega} |\varphi_k(x)|\leq \delta(\varepsilon, K,\xi)$. Thus, the strong Morrey quasiconvexity of $f$ implies that
	$$f(\xi)\leq \operatorname*{ess\,sup}_{x\in \Omega}f(\xi + D\varphi_k)+ \varepsilon,$$
	for $k$ sufficiently large. Passing to the limit on $k$ we obtain
	$$f(\xi)\leq \liminf_{k\to\infty} \operatorname*{ess\,sup}_{x\in \Omega}f(\xi+D\varphi_k)+\varepsilon.$$
	The arbitrariness of $\varepsilon$ concludes the proof.
\end{proof}

The following example shows that lower semicontinuity is not a necessary condition of functions enjoying the other convexity notions.

\begin{example}\label{exampleB}
Let $f:\mathbb{R}^{N\times n}\longrightarrow\mathbb{R}$ be such that $f=\chi_A$ where $A=\{\xi\in\mathbb{R}^{N\times n}:\ \xi_1^1\ge 1\}$ and for every $\eta \in \mathbb R^{N\times n}$, $\eta^1_1$ denotes the first entry of the matrix $\eta$. Note that $f$ is not lower semicontinuous. Indeed, considering the sequence $(\xi_k)_{k\in\mathbb{N}}\subseteq\mathbb{R}^{N\times n}$ such that $(\xi_k)_1^1=1-\frac{1}{k}$ and all the other components are zero, one has $\lim \xi_k=\xi$ where $\xi_1^1=1$, being all the other components of $\xi$ equal to zero. However,
$$f(\xi)=1>0=\lim\inf f(\xi_k).$$
On the other hand, one can easily get that $f$ is level convex as well as polyquasiconvex, weak Morrey quasiconvex, periodic-weak Morrey quasiconvex and rank-one quasiconvex, (see also Theorem \ref{thm relations between convexity notions} below).  Observe that $f$ is not strong Morrey quasiconvex because $f$ is not lower semicontinuous, (cf. Proposition \ref{fSMqcxOmegaflsc} below                                                                         ).
\end{example}

\begin{proposition}\label{fSMqcxOmegaflsc} Let $N,n\ge 1$. 
	
	\begin{itemize}
		\item[(i)] If $f:\mathbb R^{N\times n}\longrightarrow\mathbb{R}$ is strong Morrey quasiconvex in a bounded and open set $\Omega$ with Lipschitz boundary, then it is lower semicontinuous. 
\item[(ii)]	There are  functions $f:\mathbb{R}^{N\times n}\longrightarrow\mathbb{R}$ that are not lower semicontinuous but that are either level convex, polyquasiconvex, weak Morrey quasiconvex, periodic-weak Morrey quasiconvex or rank-one quasiconvex.

\end{itemize}
\end{proposition}
\begin{proof}[Proof]
The proof of (i) follows by Lemma \ref{2.5inOmegabddopen0}.	Indeed, taken a sequence $\xi_k \to 0 $ in $\mathbb R^{N \times n}$, we can define for every $x \in \Omega$, $\varphi_k(x):= \xi_k x$. This sequence of functions lying in $W^{1,\infty}(\Omega;\mathbb R^N)$ converges strongly to $0$  in $W^{1,\infty}(\Omega;\mathbb{R^N})$, hence, by the previous proposition it follows that
	$$f(\xi)\leq\liminf_{k\to \infty} f(\xi+ \xi_k).$$
	
Condition	(ii) is a consequence of Example \ref{exampleB}.
	\end{proof}

We now prove an equivalent formulation to periodic-weak Morrey quasiconvexity.

\begin{proposition}{\bf [Periodic-weak Morrey quasiconvexity.]}\label{Prop equiv periodic-weak}
Let $f:\mathbb{R}^{N\times n}\longrightarrow \mathbb{R}$ and let $C\subseteq\mathbb{R}^n$ be a cube. Then $f$ is periodic-weak Morrey quasiconvex in $C$  if and only if for every $\varphi\in W_{\rm{loc}}^{1,\infty}(\mathbb{R}^n;\mathbb{R}^N)$, such that $D\varphi$ is $C$-periodic, one has
$$\displaystyle f(\zeta)\le \operatorname*{ess\,sup}_{x\in C}f\left( D\varphi\left( x\right)  \right),$$
where $\zeta=\frac{1}{\mathcal{L}^n(C)}\int_C D\varphi(x)\,dx$.
\end{proposition}

\begin{proof}
For the non-trivial implication, let $\varphi\in W_{\rm{loc}}^{1,\infty}(\mathbb{R}^n;\mathbb{R}^N)$ be a function with $C$-periodic gradient and let $\zeta=\frac{1}{\mathcal{L}^n(C)}\int_C D\varphi(x)\,dx$. By an argument of \cite{Fonseca-Leoni-unpub}, the function defined in $C$ by
$w(x):=\varphi(x)-\zeta\cdot x$ can be extended by $C$-periodicity to an element in $W^{1,\infty}_{\rm per}(C;\mathbb{R}^N)$. Using the hypothesis of periodic-weak Morrey quasiconvexity in $C$, one gets 
$$\displaystyle f(\zeta)\le \operatorname*{ess\,sup}_{x\in C}f\left( \zeta+ Dw\left( x\right)  \right)=\operatorname*{ess\,sup}_{x\in C}f\left( D\varphi(x) \right),$$ as whished. 
\end{proof}

Next we are going to address the question of independence of domain in the notions of weak Morrey quasiconvexity and strong Morrey quasiconvexity. We observe that the class of sets that we can achieve in an equivalent notion of weak Morrey quasiconvexity is more 
general than for strong Morrey quasiconvexity. Actually, while in the first  setting (cf. Proposition \ref{lemma weak independent}) the argument relies on Vitali's covering argument, in the second one, the proof of Proposition \ref{proposition independence of domain} exploits the lower semicontinuity of the associated supremal functional as a necesssary condition to the strong Morrey quasiconvexity of the supremand (cf. Proposition \ref{Strong Morrey implies seq weak* lsc}) that involves a strong version of Besicovitch derivation theorem constraining the class of admissible sets.

\begin{proposition}\label{lemma weak independent}
	The notion of weak Morrey quasiconvexity remains unchangeable if the set $Q$ is replaced by any other bounded and open set of $\mathbb{R}^n$ with boundary of null $\mathcal{L}^n$-measure.  
\end{proposition}

\begin{proof}
	We need to show that, if $f:\mathbb{R}^{N\times n}\longrightarrow \mathbb{R}$ is a Borel measurable function such that 
	$$\displaystyle f(\xi)\le \operatorname*{ess\,sup}_{x\in O}f\left( \xi+ D\varphi\left( x\right)  \right),\ \forall\ \xi\in\mathbb{R}^{N\times n},\ \forall\ \varphi\in W_{0}^{1,\infty}(O;\mathbb{R}^N)$$ where $O$ is a given open set in $\mathbb{R}^n$, then, for any bounded open set $\Omega$ of $\mathbb{R}^n$ whose boundary has null $\mathcal L^n$-measure, one has
	$$\displaystyle f(\xi)\le \operatorname*{ess\,sup}_{x\in \Omega}f\left( \xi+ D\psi\left( x\right)  \right),\ \forall\ \xi\in\mathbb{R}^{N\times n},\ \forall\ \psi\in W_{0}^{1,\infty}(\Omega;\mathbb{R}^N).$$
	
	Let $O$ and $f$ be as above and let $\Omega$ be a bounded and open set with $\mathcal{L}^n(\partial\Omega)=0$. Let $\xi\in\mathbb{R}^{N\times n}$ and $\psi\in W_{0}^{1,\infty}(\Omega;\mathbb{R}^N).$ Without loss of generality, assume $\Omega$ is connected, otherwise consider one of its connected components.
	
	Let $x_0\in\Omega$ and define $\Omega_0:=\Omega-x_0$. Consider, $\mathcal{G}$, the collection of open sets $a+\varepsilon\,\Omega_0$, for $a\in\mathbb{R}^n$ and $\varepsilon>0$. By Vitali covering theorem (see \cite[Corollary 10.5]{DacMar}), up to a set of measure zero, the set $O$ can be covered with a countable number of sets $G\in\mathcal{G}$ with disjoint closures. More precisely, for some countable collection $\mathcal{G}'\subseteq\mathcal{G}$,  
	$$\bigcup_{G\in\mathcal{G}'}G\subseteq O,\qquad \mathcal{L}^n\left(O\setminus\bigcup_{G\in\mathcal{G}'}G\right)=0,$$ and $G\cap F=\emptyset$ for $G,F\in \mathcal{G}'$ with $G\neq F$.
	
	Each set $G\in\mathcal{G}'$ has the form $a+\varepsilon\,\Omega_0$. On each of these sets, define a function $\psi_{a,\varepsilon}$ as $$\psi_{a,\varepsilon}(y)=\varepsilon\psi\left(\frac{y-a}{\varepsilon}+x_0\right).$$ Observe that $$D\psi_{a,\varepsilon}(y)=D\psi\left(\frac{y-a}{\varepsilon}+x_0\right).$$ Patching these functions together, we construct a function $\varphi$ defined on $O$ as 
	$\varphi=\psi_{a,\varepsilon}$ in each $a+\varepsilon\,\Omega_0\in\mathcal{G}'$, and $\varphi=0$ in $O\setminus \cup_{G\in\mathcal{G}'}G$. In this way, one gets $\varphi\in W_{0}^{1,\infty}(O;\mathbb{R}^N)$, and thus, using the hypothesis, we get
	$$\displaystyle f(\xi)\le \operatorname*{ess\,sup}_{y\in O}f\left( \xi+ D\varphi\left( y\right)  \right)\le \operatorname*{ess\,sup}_{x\in \Omega}f\left( \xi+ D\psi\left( x\right)  \right),$$ as desired.
\end{proof}

Proposition \ref{lemma weak independent} provides an answer to the question raised in the introduction regarding the minimization of some supremal functionals on a set of functions with a prescribed affine boundary condition. This is stated in the next corollary that is an immediate consequence of the previous result.

\begin{corollary}\label{Corollary 3.8}
	Let $\Omega$ be a bounded open set with boundary of null $\mathcal{L}^n$-measure and let $f:\mathbb{R}^{N\times n}\longrightarrow\mathbb{R}$ be a Borel measurable function. Consider the functional $$I(u)=\operatorname*{ess\,sup}_{x\in \Omega} f(Du(x)), \text{ for } u\in W^{1,\infty}(\Omega;\mathbb{R}^N).$$ Let $\xi\in\mathbb{R}^{N\times n}$ and denote by $u_\xi$ an affine function with gradient $\xi$.
	
	If $f$ is weak Morrey quasiconvex, then $u_\xi$ minimizes $I$ on $u_\xi+W_0^{1,\infty}(\Omega;\mathbb{R}^N)$. 
\end{corollary}
 
Next we address the invariance of domain in the notion of strong Morrey quasiconvexity. We start with translation of sets.

\begin{remark}\label{StrongMorreytrans}
	If $f:\mathbb{R}^{N\times n}\longrightarrow \mathbb{R}$ is strong Morrey quasiconvex in a bounded open set $\Omega\subseteq\mathbb{R}^n$ 
	then, it is also strong Morrey quasiconvex in any translation of $\Omega$. That is, for any $x_0\in\mathbb{R}^n$
	\begin{equation}\label{translfSMQcx}
		\forall\ \varepsilon>0\ \forall\ \xi\in \mathbb{R}^{N\times n}\ \forall\ K>0\ \exists\ \delta=\delta(\varepsilon, K,\xi)>0:$$
		$$\left.\begin{array}{l}\varphi\in W^{1,\infty}(x_0+\Omega;\mathbb{R}^N)\vspace{0.2cm}\\ ||D\varphi||_{L^\infty(x_0+\Omega;\mathbb{R}^{N\times n})}\le K\vspace{0.2cm}\\ \max_{x\in\partial (x_0+\Omega)}|\varphi(x)|\le \delta\end{array}\right\}\Longrightarrow f(\xi)\le \operatorname*{ess\,sup}_{x\in x_0+\Omega} f(\xi+D\varphi(x))+\varepsilon,
	\end{equation}
	where $x_0+\Omega$ denotes the translation of the set $\Omega$ by the vector $x_0$. Observe that $\partial (x_0+\Omega)=x_0+\partial \Omega$ and that $x_0 +\Omega$ is also bounded and open. Also, observe that for every function $\varphi \in W^{1,\infty}(x_0+\Omega;\mathbb R^N)$, defining $\psi(x):= \varphi(x_0+x)$ for $x \in \Omega$, one has $\psi \in W^{1,\infty}(\Omega;\mathbb R^N)$.
	If $\|D \varphi\|_{L^\infty(x_0+\Omega;\mathbb{R}^{N\times n})}\leq K$, it results
	$$\|D \psi\|_{L^\infty(\Omega;\mathbb{R}^{N\times n})}\|=\|D \varphi\|_{L^\infty(x_0+\Omega;\mathbb{R}^{N\times n})}\leq K,$$
	and, if $\max_{x\in\partial (x_0+\Omega)}|\varphi(x)|\le \delta$, then
	$$\max_{x \in \partial \Omega}|\psi(x)|=\max_{x \in \partial(x_0+\Omega)}|\varphi(x)|\leq \delta.$$
	Thus,  by \eqref{SMQcxOmega}, we have for every $\varepsilon>0,$ $\xi\in\mathbb R^{N\times n}$ , $K>0$ that there exists a $\delta\equiv \delta(\varepsilon, K, \xi)>0$ such that
	$$f(\xi)\leq\operatorname*{ess\,sup}_{x\in \Omega} f(\xi+D\psi(x))+\varepsilon= \operatorname*{ess\,sup}_{x\in x_0+\Omega} f(\xi+D\varphi(x))+\varepsilon,$$
	thus proving \eqref{translfSMQcx}.
\end{remark}  

Our next goal is to address more general cases. On the one hand, it is hard to deal with the strong Morrey quasiconvexity notion directly. But, on the other hand, we can relate it with lower semicontinuity of supremal functionals, independently of their domain of definition, (see Propositions \ref{Strong Morrey implies seq weak* lsc} and \ref{seq weak* lsc implies strong Morrey} in the Appendix). Therefore to achieve our goal we pass through properties of supremal functionals.


\begin{proposition}\label{proposition independence of domain}
	In the notion of strong Morrey quasiconvexity, the set $Q$ can be replaced by any other bounded, open, and convex set of $\mathbb{R}^n$.
\end{proposition}
\begin{proof}[Proof] Assume that $f$ is strong Morrey quasiconvex in a bounded, convex, and open set $\Omega\subseteq\mathbb{R}^n$. By Proposition \ref{Strong Morrey implies seq weak* lsc} one has that $F(u,O):=\operatorname*{ess\,sup}_{x\in \Omega}f\left(Du\left( x\right)  \right)$ is sequentially weakly* l.s.c. in $W^{1,\infty}(O;\mathbb{R}^N)$ for any bounded and open set $O\subseteq\mathbb{R}^n$. Then it is enough to invoke Proposition \ref{seq weak* lsc implies strong Morrey} to conclude that $f$ is also strong Morrey quasiconvex in $O$. 
	\hfill
\end{proof}

As a side result, we also get that the sequential weak* lower semicontinuity of $F$ is independent of the domain $\Omega$ in the class of bounded, open and convex sets. 

\begin{proposition}\label{proposition independence for functionals} Let $f:\mathbb{R}^{N\times n}\longrightarrow\mathbb{R}$ be a Borel measurable function and consider the functional
	$$F(u,\Omega):=\operatorname*{ess\,sup}_{x\in \Omega}f\left(Du\left( x\right) \right),$$ where $\Omega$ is a bounded open set of $\mathbb{R}^n$ and $u\in W^{1,\infty}(\Omega;\mathbb{R}^N)$. 
	
If $F(\cdot,\mathcal{O}_1)$ is sequentially weakly* lower semicontinuous in $W^{1,\infty}(\mathcal{O}_1;\mathbb{R}^N)$ with $\mathcal{O}_1$ bounded, convex, and open, then $F(\cdot,\mathcal{O}_2)$ is sequentially weakly* lower semicontinuous in $W^{1,\infty}(\mathcal{O}_2;\mathbb{R}^N)$ for any bounded open set $\mathcal{O}_2$. 
	
In particular, the sequential weak* lower semicontinuity of $F(\cdot,\Omega)$ is independent of the domain $\Omega$ in the class of bounded, convex, and open sets.
\end{proposition}

\begin{proof}
	The result follows directly from the Propositions \ref{Strong Morrey implies seq weak* lsc} and \ref{seq weak* lsc implies strong Morrey}.
\end{proof}

%
%
%


\medskip

\subsection{Hierarchy of convexity notions.}\label{subsection Hierarchies}

The convexity notions related to the lower semicontinuity of the supremal functionals under consideration having been introduced, we investigate in the sequel how they relate to each other. We retake the work by \cite{BJW} and we try to make an exhaustive study of the notions of convexity introduced above in terms of necessary and sufficient conditions to each of them. We review the properties stated therein, we establish other relations and we provide counter-examples whenever possible. The section finishes with a list of questions that remain open.


\begin{theorem}\label{thm relations between convexity notions}
Let $N, n \in \mathbb N$ and let $f:\mathbb{R}^{N\times n}\longrightarrow \mathbb{R}$.
\begin{enumerate} 
\item If $f$ is level convex then $f$ is polyquasiconvex and rank one quasiconvex. 

\noindent If $f$ is also Borel measurable then $f$ is weak and periodic-weak Morrey quasiconvex. 

\noindent Moreover, if $f$ is level convex and lower semicontinuous, then $f$ is strong Morrey quasiconvex.
\item Assume that $f$ is polyquasiconvex and satisfies one of the following hypotheses:
\begin{enumerate} 
\item[(i)] $f=g\circ T$, where $T$ is as in the definition of polyquasiconvexity and $g:\mathbb{R}^{\tau(n,N)}\longrightarrow\mathbb{R}$ is level convex and l.s.c.;
\item[(ii)] $f$ is l.s.c. and $\lim_{|\xi|\to+\infty}f(\xi)=+\infty$, 
\end{enumerate}
then $f$ is strong Morrey quasiconvex.

\noindent If $f$ is polyquasiconvex and $f=g\circ T$, with $g$ level convex and Borel measurable, then $f$ is weak and periodic-weak Morrey quasiconvex.

\noindent If $f$ is polyquasiconvex then $f$ is rank-one quasiconvex.
\item  If $f$ is  strong Morrey quasiconvex then $f$ is weak and  periodic-weak Morrey quasiconvex.
\item \label{(4)} If $f$ is periodic-weak Morrey quasiconvex then $f$ is weak Morrey quasiconvex.
\item If $f$ is weak Morrey quasiconvex and upper semi-continuous then $f$ is rank-one quasiconvex.
\end{enumerate}

\begin{enumerate} 
\setcounter{enumi}{5}
\item If $f$ is periodic-weak Morrey quasiconvex in any cube $C\subseteq\mathbb{R}^n$ then $f$ is rank-one quasiconvex. In particular, if $f$ is strong Morrey quasiconvex then $f$ is rank-one quasiconvex.
\item Let $n=1$ or $N=1$. Then $f$ is level convex if and only if it is polyquasiconvex and if and only if it is rank-one quasiconvex. Furthermore $(i)$ if $f$ is lower semicontinuous, then $f$ is level convex if and only if it is strong Morrey quasiconvex, $(ii)$ if $f$ is upper semicontinuous then $f$ is level convex if and only if $f$ is weak Morrey quasiconvex, and if and only if it is periodic-weak Morrey quasiconvex. In particular, if $f$ is continuous, all the notions are equivalent. 

If $n=1$, the upper semicontinuity can be replaced by Borel measurability, to get $f$ level convex if and only if 
$f$ weak Morrey quasiconvex and if and only $f$ periodic-weak Morrey quasiconvex. 

\end{enumerate} 
\end{theorem}

To ease the reading of the theorem, consider the following figure.
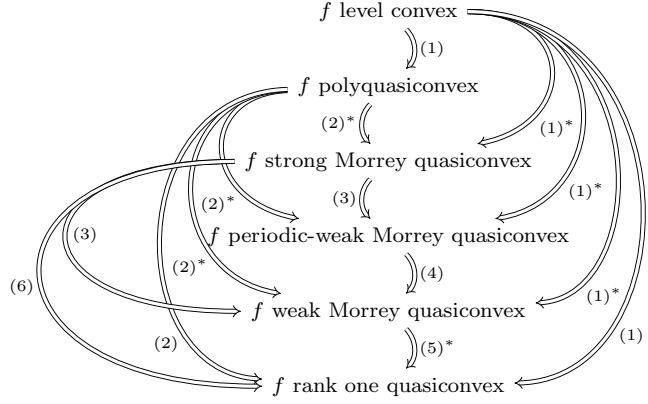
\begin{figure}[H]
\begin{tikzcd}[arrows=Rightarrow]
{f \text{ level convex}} \arrow[out=-45,in=45]{d}{}[right]{(1)} 
\arrow[dd, "(1)^*", to path={[pos=0.75]..controls +(2.5,0) and  +(2.5,0.5).. (\tikztotarget) \tikztonodes}] 
\arrow[ddd, "(1)^*", to path={[pos=0.75]..controls +(3,0) and  +(3,0.5).. (\tikztotarget) \tikztonodes}] 
\arrow[dddd, "(1)^*", to path={[pos=0.85]..controls +(3.5,0) and  +(3.5,0.2).. (\tikztotarget) \tikztonodes}] 
\arrow[ddddd, "(1)", to path={[pos=0.75]..controls +(4,0) and  +(4,0.1).. (\tikztotarget) \tikztonodes}] 
\\
{f \text{ polyquasiconvex}} \arrow[out=-135,in=135]{d}{}[left]{(2)^*} 
\arrow[dd, "(2)^*", to path={[swap, pos=0.75]..controls +(-2.5,-0.1) and  +(-2.5,0.5).. (\tikztotarget) \tikztonodes}] 
\arrow[ddd, "(2)^*", to path={[swap, pos=0.75]..controls +(-3,-0.1) and  +(-3,0.5).. (\tikztotarget) \tikztonodes}] 
\arrow[dddd, "(2)", to path={[swap, pos=0.75]..controls +(-3.5,-0.1) and  +(-3.5,0.2).. (\tikztotarget) \tikztonodes}] 
\\
{f \text{ strong Morrey quasiconvex}} \arrow[out=-135,in=135]{d}{}[left]{(3)} 
\arrow[dd, "(3)", to path={[pos=0.55]..controls +(-5,0) and  +(-5,0).. (\tikztotarget) \tikztonodes}] 
\arrow[ddd, "(6)", to path={[swap, pos=0.5]..controls +(-5.5,0) and  +(-5.5,0).. (\tikztotarget) \tikztonodes}] 
\\
{f \text{ periodic-weak Morrey quasiconvex}} \arrow[out=-45,in=45]{d}{}[right]{(4)}  \\
{f \text{ weak Morrey quasiconvex}} \arrow[out=-45,in=45]{d}{}[right]{(5)^*} \\
{f \text{ rank one quasiconvex}}
\end{tikzcd}
\caption{The figure shall be considered accompanied by Theorem \ref{thm relations between convexity notions}. Namely, the labels to each implications refer to the items in the theorem containing the referred implication and, when additional hypotheses are required, the label is signed with $^*$.}
\end{figure}

\begin{remark}
 \begin{enumerate}
\item As observed earlier, strong Morrey quasiconvex functions are lower semicontinuous while polyquasiconvex functions may fail to enjoy this property. For that reason we considered in (2) of the previous proposition lower semicontinuity assumptions.
As we will see in Proposition \ref{proposition-counter-implications} the upper semicontinuity hypothesis cannot be removed in (5).
\item With respect to (6), observe that if we only assume $f$ is periodic-weak Morrey quasiconvex (in the cube $Q$) then we can only get the rank-one quasiconvexity of $f$ in some rank one directions, namely those that are given by matrices with only one non-null column. Note that this is what is established by Ansini and Prinari in \cite[Proposition 5.1 (i)]{APSIMA} in the case of the ${\rm curl}$ operator. Actually, as already mentioned, $\rm curl$-weak quasiconvexity in \cite{APSIMA} is what we called periodic-weak Morrey quasiconvexity. According to \cite[Proposition 5.1 (i)]{APSIMA}, this ensures the level convexity inequality in the directions of the kernel of the  ${\rm curl}$ operator that are the directions that we find in our argument ({\sl cf.} Remark \ref{remark periodic-weak} (iii)).
\item[(3)] Contrary to the case $n=1$, where weak Morrey quasiconvexity and periodic-weak Morrey quasiconvexity imply level convexity, we shall see in Proposition \ref{proposition-counter-implications}, that in the case $N=1$, this is not true for weak Morrey quasiconvexity,  even if the function $f$ is lower semicontinuous, while it is currently open the periodic-weak Morrey setting.
\end{enumerate}
\end{remark}

\begin{proof}
Conditions (1), (3), (4) and (6) follow from standard arguments, as well as the last two assertions of (2). For the first statement of (1), it is enough to make use of the first component of the vector function $T$. The second assertion of (1) follows by restricting to rank-one connected matrices. The next two assertions of (1) follow from Jensen's inequality (cf. Theorem \ref{Jensensupremalscalar}) 
applied with $\varphi= \xi+ D\psi$, for $\psi\in W_0^{1,\infty}(Q;\mathbb{R}^N)$ or $\psi\in W_{\rm per}^{1,\infty}(Q;\mathbb{R}^N)$,  $\Omega= Q$, and $\mu$ the Lebesgue measure restricted to the cube $Q$ if we observe that, either $W_0^{1,\infty}(Q;\mathbb{R}^N)$ and $W_{\rm per}^{1,\infty}(Q;\mathbb{R}^N)$ have zero integral average. 
 The last assertion of (1) follows from the previous ones and (2) (i), once this is proved.

With respect to (3), given $\varphi\in W_{\rm{per}}^{1,\infty}(Q;\mathbb{R}^N)$, 
define $\varphi_n(x):=\frac{1}{n}\varphi(nx)$. Then let $\varepsilon=\frac{1}{n}$, $K=||D\varphi||_{L^{\infty}(Q;\mathbb{R}^{N\times n})}$, and $\xi\in\mathbb{R}^{N\times n}$. Consider the constant $\delta>0$ provided by the assumption of strong Morrey quasiconvexity. Note that, for sufficiently large $n$, $\max_{x\in\partial Q}|\varphi_n(x)|\le \delta$. Therefore, applying the assumption, one gets $$f(\xi)\le \operatorname*{ess\,sup}_{x\in Q} f(\xi+D\varphi_n(x))+\frac{1}{n}=\operatorname*{ess\,sup}_{x\in Q} f(\xi+D\varphi(x))+\frac{1}{n},$$ and the desired inequality is achieved by letting $n\to\infty$. Regarding (4), it is enough to observe that $W_{0}^{1,\infty}(Q;\mathbb{R}^N)\subseteq W_{\rm{per}}^{1,\infty}(Q;\mathbb{R}^N)$. With respect to second assertion in (2), let $\xi\in\mathbb{R}^{N\times n}$ and $\varphi\in W_{\rm per}^{1,\infty}(Q;\mathbb{R}^N)$. Since $\mathrm{adj}_s$, $2\le s\le \min\{n,N\}$, are quasiaffine functions (in the sense of \cite[Definition 1.5]{Dbook}), $$T(\xi)=\int_QT(\xi+D\varphi(x))\,dx$$ and thus, by Theorem \ref{Jensensupremalscalar},
$$\displaystyle f(\xi)=g(T(\xi))=g\left(\int_QT(\xi+D\varphi(x))\,dx\right)\le \operatorname*{ess\,sup}_{x\in Q}g\left( T(\xi+ D\varphi\left( x\right) ) \right)= \operatorname*{ess\,sup}_{x\in Q}f\left( \xi+ D\varphi\left( x\right)  \right).$$ In this way we proved that $f$ is periodic-weak Morrey quasiconvex. Next, we prove the last assertion of (2). Let $\xi,\eta\in\mathbb{R}^{N\times n}$ such that ${\rm rank}(\xi-\eta)=1$. Then, for some level convex function $g:\mathbb{R}^{\tau(n,N)}\longrightarrow\mathbb{R}$, $f=g\circ T$ and 
\begin{align*}
f(\lambda\xi+(1-\lambda)\eta)= & 
g(T(\lambda\xi+(1-\lambda)\eta))
=g(\lambda T(\xi)+(1-\lambda)T(\eta))\\
\leq & \max\{g(T(\xi)),g(T(\eta))\}
=\max\{f(\xi),f(\eta)\}
\end{align*} 
where we have used \cite[Lemma 5.5]{Dbook} and the level convexity of $g$, achieving the rank-one quasiconvexity of $f$. Finally, we prove (6). It follows as in \cite[proof of Theorem 7.7 (ii)]{Dbook}. Let $\xi,\eta\in\mathbb{R}^{N\times n}$ be such that ${\rm rank}(\xi-\eta)=1$. Then $\xi-\eta=a\otimes\nu$ for some $a\in\mathbb{R}^N$ and $\nu\in\mathbb{R}^n$ is a unit vector. Let $R\in \mathcal{SO}(n)$ be a special orthogonal matrix such that $R e_1=\nu$, where $e_1$ is the first vector of the canonical basis of $\mathbb{R}^n$, and denote by $C$ the cube $RQ$. Then, we can construct a function $\varphi\in W^{1,\infty}_{\rm per}(C;\mathbb{R}^N)$ such that $D\varphi\in\{(1-\lambda)(\xi-\eta),-\lambda(\xi-\eta)\}$ {\sl a.e.} in $C$. Therefore, applying the periodic-weak assumption on $f$ in every cube, one gets
$$f(\lambda\xi+(1-\lambda)\eta)\le \operatorname*{ess\,sup}_{x\in C}f(\lambda\xi+(1-\lambda)\eta+D\varphi)=\max\{f(\xi),f(\eta)\},$$ proving the rank-one quasiconvexity. The statement regarding strong Morrey quasiconvexity, follows from Proposition \ref{proposition independence of domain} combined with the previous.

Condition (5) was proved in \cite[Theorem A.5]{RZ}. Regarding (7), it suffices to observe that level convexity is equivalent to rank-one quasiconvexity and polyquasiconvexity. All the previous points guarantee the remaining equivalences, up to the last assertion, which follows by standard arguments.

It remains to prove the first part of (2). 

This relies on results regarding lower semicontinuity of functionals presented in the Appendix.
By Proposition \ref{seq weak* lsc implies strong Morrey} combined with Proposition \ref{proposition independence of domain}, it is enough to show that, under each of the two set of hypotheses, the functional $F(u,O):=\operatorname*{ess\,sup}_{x\in \Omega}f\left(Du\left( x\right)  \right)$ is sequentially weakly* l.s.c. in $W^{1,\infty}(O;\mathbb{R}^N)$. 

First we present the proof of the sequential weak* l.s.c. of the functional $F$ under assumption (i). Let $(u_k)_{k\in\mathbb{N}}\subseteq W^{1,\infty}(O,\mathbb{R}^N)$ be an arbitrary sequence weakly* converging to some function $u$ in $W^{1,\infty}(O,\mathbb{R}^N)$. We want to show that 
$$F(u,O) \le \liminf_{k\to\infty}F(u_k,O).$$
Since, by \cite[Theorem 8.20, Remark 8.21 (iv)]{Dbook}, $T(Du_k)$ weakly* converges to $T(Du)$ in $L^\infty(O;\mathbb{R}^{\tau(n,N)})$, it is enough to show that 
$$G(V,O):=\operatorname*{ess\,sup}_{x\in O}g\left(V\left( x\right)  \right),\qquad V\in L^\infty(O;\mathbb{R}^{\tau(n,N)})$$
is sequential weak* l.s.c. in $L^\infty(O;\mathbb{R}^{\tau(n,N)}).$ Let then $(V_k)_{k\in\mathbb{N}}\subseteq L^\infty(O;\mathbb{R}^{\tau(n,N)})$ be an arbitrary sequence weakly* converging in $L^\infty(O;\mathbb{R}^{\tau(n,N)})$ to some function $V$.

Let $r:=\liminf_{k\to\infty}G(V_k,O)=\lim_{i\to\infty}G(V_{k_i},O)$ for some subsequence $(V_{k_i})_{i\in\mathbb{N}}$ of $(V_k)_{k\in\mathbb{N}}$. Then, by definition of limit, for arbitrary $\varepsilon>0$, there is $i_0\in\mathbb{N}$ such that, for $i\ge i_0$, 
$$g(V_{k_i}(x))\le \operatorname*{ess\,sup}_{x\in O}g\left(V_{k_i}\left(x\right)\right)=G(V_{k_i},O)\le r+\varepsilon\qquad\text{for}\ a.e.\ x\in O.$$
That is, denoting $E_{r+\varepsilon}:=\{S\in \mathbb{R}^{\tau(n,N)}:\ g(S)\le r+\varepsilon \}$
one has, for $i\ge i_0$, 
$V_{k_i}(x)\in E_{r+\varepsilon}\ \text{for}\ a.e.\ x\in O$ and thus $\mathrm{d}(V_{k_i}(x),E_{r+\varepsilon})=0\ \text{for}\ a.e.\ x\in O,$ where $\mathrm{d}(\cdot,E_{r+\varepsilon})$ denotes the distance function to the set $E_{r+\varepsilon}$.
Since $g$ is level convex, the set $E_{r+\varepsilon}$ is convex and by \cite[Theorem 5.14]{Fonseca-Leoni-Lp} the functional 
$$D(V,O):=\int_O\mathrm{d}\left(V(x);E_{r+\varepsilon}\right)\,dx$$ is sequentially weakly* l.s.c. in $L^\infty(O;\mathbb{R}^{\tau(n,N)})$. Therefore, 
$$0\le \int_O\mathrm{d}\left(V(x);E_{r+\varepsilon}\right)\,dx\le \liminf_{i\to\infty}\int_O\mathrm{d}\left(V_{k_i}(x);E_{r+\varepsilon}\right)\,dx=0$$
and $\mathrm{d}\left(V(x);E_{r+\varepsilon}\right)$ for $a.e.\ x\in O$. Using the hypothesis that $g$ is l.s.c. we have that $E_{r+\varepsilon}$ is closed and thus $V(x)\in E_{r+\varepsilon}$ for $a.e.\ x\in O$ that gives 
$$\operatorname*{ess\,sup}_{x\in O}g\left(V\left( x\right)  \right)\le r+\varepsilon$$
ensuring the desired condition by letting $\varepsilon\to 0$ and recalling the definition of $r$.

Finally, we prove the sequential weak* l.s.c. of the functional $F$ under condition (ii). As before, let $(u_k)_{k\in\mathbb{N}}\subseteq W^{1,\infty}(O,\mathbb{R}^N)$ be an arbitrary sequence weakly* converging to some function $u$ in $W^{1,\infty}(O,\mathbb{R}^N)$. Let $(u_{k_i})_{i\in\mathbb{N}}$ be a subsequence of $(u_k)_{k\in\mathbb{N}}$ such that $\liminf_{k\to\infty}F(u_k,O)=\lim_{i\to\infty}F(u_{k_i},O)$ and let $r:=\lim_{i\to\infty}F(u_{k_i},O)$. Defining $E_{r+\varepsilon}:=\{\xi\in \mathbb{R}^{N\times n}:\ f(\xi)\le r+\varepsilon \}$, one has that, given $\varepsilon>0$, there is $i_0\in\mathbb{N}$ such that, for $i\ge i_0$, 
$$T(Du_{k_i}(x))\in T(E_{r+\varepsilon})\ \text{for}\ a.e.\ x\in O.$$ In particular, $\mathrm{d}(T(Du_{k_i}(x)),T(E_{r+\varepsilon}))=0\ \text{for}\ a.e.\ x\in O$ and also
$$\mathrm{d}(T(Du_{k_i}(x)),\mathrm{co}(T(E_{r+\varepsilon})))=0\ \text{for}\ a.e.\ x\in O.$$ Now we invoke, as above, the sequential weak* l.s.c. in $L^\infty(O;\mathbb{R}^{\tau(n,N)})$ of the functional $$B(V,O):=\int_O\mathrm{d}\left(V(x);\mathrm{co}(T(E_{r+\varepsilon}))\right)\,dx.$$
Since, $T(Du_{k_i})$ weakly* converges to $T(Du)$ in $L^\infty(O;\mathbb{R}^{\tau(n,N)})$,
$$0\le \int_O\mathrm{d}\left(T(Du(x));\mathrm{co}(T(E_{r+\varepsilon}))\right)\,dx\le \liminf_{i\to\infty}\int_O\mathrm{d}\left(T(Du_{k_i}(x));\mathrm{co}(T(E_{r+\varepsilon}))\right)\,dx=0$$
giving 
\begin{equation}\label{eq distance to convex}
	\mathrm{d}\left(T(Du(x));\mathrm{co}(T(E_{r+\varepsilon}))\right)=0\ \text{for }a.e.\ x\in O.
\end{equation}

Since $f$ is l.s.c., the set $E_{r+\varepsilon}$ is closed. Moreover, the growth assumption on $f$, ensures that $E_{r+\varepsilon}$ is  bounded and thus compact. Therefore, $T(E_{r+\varepsilon})$ is also compact and we can apply \cite[Theorem 2.14]{Dbook} to ensure that $\mathrm{co}(T(E_{r+\varepsilon}))$ is closed. 
This, together with \eqref{eq distance to convex} entails that 
$$T(Du(x))\in \mathrm{co}(T(E_{r+\varepsilon}))\ \text{for }a.e.\ x\in O.$$

It is now enough to show that 
\begin{equation}\label{eq polyconvex set}
	\{\xi\in\mathbb{R}^{N\times n}:\ T(\xi)\in \mathrm{co}(T(E_{r+\varepsilon}))\}=E_{r+\varepsilon}
\end{equation}
to conclude that $Du(x)\in E_{r+\varepsilon}\ \text{for }a.e.\ x\in O$ that, in turn, ensures $f(Du(x))\le r\ \text{for }a.e.\ x\in O$ as wished.  

Regarding \eqref{eq polyconvex set}, it follows from \cite[Theorem 7.4 (iii)]{Dbook} and the fact that $E_{r+\varepsilon}$ is polyconvex in the sense of \cite[Definition 7.2 (ii)]{Dbook}, by \cite[Theorem 7.4 (ii)]{Dbook} and the polyquasiconvexity of $f$. \hfill 
\end{proof}

We give, next, several examples of functions enjoying or not the convexity notions discussed above. These examples, besides the interest in itself, will be useful to discuss in Proposition \ref{proposition-counter-implications} below the validity of the counter-implications of the previous proposition.

\begin{example}\label{exampleA}
Let $g:\mathbb{R}\longrightarrow\mathbb{R}$ be the characteristic function $g=\chi_{]1,\infty)}$ and, for $n >1$, define $f:\mathbb{R}^{n\times n}\longrightarrow\mathbb{R}$ as $f(\xi)=g(\det(\xi))$. Trivially, since $g$ is level convex, $f$ is polyquasiconvex. However, one can easily see that $f$ is not level convex. Moreover, $f$ is lower semicontinuous because $g$ is lower semicontinuous and the determinant is a continuous function.
\end{example}

\begin{example}\label{exampleC} This example was given in \cite[Example A.3]{RZ}. Exploiting it further, it also serves to discuss periodic-weak Morrey quasiconvexity. Let $N\ge 1$ and $n>1$. Let $S:=\{\xi, \eta\}\subset\mathbb{R}^{N\times n}$ such that $\rm{rank}(\xi-\eta)=1$ and let $f:=1-\chi_S$, where $\chi_S$ is the characteristic function of $S$. As proved in \cite[Example A.3]{RZ}, the function $f$ is not rank one quasiconvex, but it is weak Morrey quasiconvex. In particular, as noticed in \cite[Example 2.7]{P}, $f$ is not strong Morrey quasiconvex. Note also, that $f$ is lower semicontinuous (although it is not continuous).

Moreover, if we choose $\xi$ and $\eta$ such that $\xi-\eta=a\otimes e_1$ for some $a\in \mathbb{R}^N$ ($e_1$ being the first vector of the canonical basis in $\mathbb{R}^n$), arguing as in \cite[proof of Theorem 3.2 (ii)]{DR} (see also \cite[page 319]{Dbook}) we conclude that $f$ is not periodic-weak Morrey quasiconvex. (Note that here we need to construct a function $\varphi\in W_{\rm{per}}^{1,\infty}(Q;\mathbb{R}^N)$ and that is the reason to choose $\xi$ and $\eta$ so that $\xi-\eta$ is compatible with the cube $Q$.) In an analogous way, considering appropriate matrices $\xi$ and $\eta$ we can ensure that $f$ is not periodic-weak Morrey quasiconvex in a given cube $C$.
\end{example}

\begin{example}\label{exampleE} According to the result proved by Kirchheim \cite{Kirchheim} (see also \cite[Theorem 7.12]{Dbook}), if $N\ge 2$ and $n\ge 2$, there is a finite number of $N\times n$ matrices, $\xi_1,...,\xi_m\in\mathbb{R}^{N\times n}$, such that
$\operatorname*{rank}(\xi_i-\xi_j)>1,\ \forall\ i\neq j$ and
there exist $\xi_0\notin\{\xi_1,...,\xi_m\}$ and $u\in u_{\xi_0}+W_0^{1,\infty}(Q;\mathbb{R}^N)$
(where $u_{\xi_0}$ denotes an affine function verifying $Du_{\xi_0}(x)\equiv\xi_0$) with $Du(x)\in \{\xi_1,...,\xi_m\},\
a.e.$ in $Q$. Consider then the function $f=1-\chi_S$ where $S=\{\xi_1-\xi_0,...,\xi_m-\xi_0\}$ and $\chi_S$ is the characteristic function of $S$. Of course $f$ is lower semi-continuous and, by the properties stated above it is rank-one quasiconvex, but not strong Morrey quasiconvex. To show this last statement, it's enough to consider $\varphi:=u-u_{\xi_0}\in W_0^{1,\infty}(Q;\mathbb{R}^N)$ to get a contradiction to strong Morrey quasiconvexity. Indeed, take $\varepsilon\in(0,1)$, $\xi=0$, and $K=\max\{|\xi_0-\xi_1|,...,|\xi_0-\xi_m|\}$ and observe that $f(0)=1>\varepsilon= \operatorname*{ess\,sup}_{x\in Q}f(0+D\varphi(x))+\varepsilon$.
\end{example}

\begin{proposition}\label{proposition-counter-implications}
Let $N,n\in\mathbb{N}$ and denote by $f$ a real valued function defined in $\mathbb{R}^{N\times n}$. 
\begin{itemize}
\item[(i)] If $N, n>1$, there are (lower semicontinuous) non-level convex functions $f$ that are polyquasiconvex, strong Morrey quasiconvex, periodic-weak Morrey quasiconvex, weak Morrey quasiconvex, and rank one quasiconvex.
\item[(ii)] For $N\ge 1$ and $n>1$, there are (lower semicontinuous) functions $f$ that are weak Morrey quasiconvex, but neither polyquasiconvex, nor strong Morrey quasiconvex, nor rank one quasiconvex, nor periodic-weak Morrey quasiconvex in a fixed cube $C$. In particular, taking $C=Q$, there are non periodic-weak Morrey quasiconvex functions satisfying all the previous properties.
\item [(iii)] If $N, n >1$, there are (lower semicontinuous) rank one quasiconvex functions $f$ that are not strong Morrey quasiconvex.
\end{itemize}
\end{proposition}

\begin{proof}
Statement (i) is proved by Example \ref{exampleA} having in mind the implications (2) (i), (7), (4), and (8), stated in Theorem \ref{thm relations between convexity notions}. Observe that this example can be easily adapted to the case $N \not = n$. Statement (ii) is proved by Example \ref{exampleC} having in mind Theorem \ref{thm relations between convexity notions} (8) and (9).
Finally, statement (iii) is proved by Example \ref{exampleE}.
\end{proof}



\color{black}
\begin{example}
	\label{arctanex}
	If $f:\mathbb R^{N\times n}\to \mathbb R$ is a lower semicontinuous polyquasiconvex function, satisfying $\lim_{|\xi|\to +\infty} $ $f(\xi)=+\infty$, then the bounded function $\arctan(f):\mathbb R^{N\times n}\to (-\pi/2,\pi/2)$ is polyquasiconvex and strong Morrey quasiconvex. 
\end{example}

The previous analysis leaves open several questions that we list below.
\begin{enumerate}
\item Example \ref{arctanex} shows that the assumptions of Theorem \ref{thm relations between convexity notions} (2)(ii) are not sharp. We can wonder if the coercivity condition $\lim_{|\xi|\to +\infty} f(\xi)=+\infty$ can be removed in general. 
Remind also that if we assume $f$ is level convex and lower semicontinuous (that are in particular polyquasiconvex) then $f$ is strong Morrey quasiconvex with no need of any growth assumption.

\item Can we get an example showing that strong Morrey quasiconvexity does not imply polyquasiconvexity? Recall that in the integral setting there are examples of quasiconvex functions which are not polyconvex (\cite[Theorem 5.51]{Dbook}.

\item Does periodic-weak Morrey quasiconvexity imply polyquasiconvexity?
\item Does periodic-weak Morrey quasiconvexity in any cube $C$ together with lower semicontinuity imply strong Morrey quasiconvexity? (This being the case, then the two conditions are equivalent.)
\item Does weak Morrey quasiconvexity together with the continuity of the function imply strong Morrey quasiconvexity?
\item The results of next section suggest that necessary and sufficient conditions may be obtained under a coercivity assumption. In particular does weak Morrey quasiconvexity imply rank-one or strong Morrey quasiconvexity in the class of coercive functions?
\end{enumerate}

Questions $(3), (4)$ are open even in the scalar case $N=1$.


\section{Convexity notions arising in connection with power-law approximation}
 
In this section, we address the comparison between definitions of the previous sections and those arising in the context of so-called power-law approximation, in particular the notions of ${\rm curl}-\infty$ quasiconvexity and ${\rm curl}$-Young quasiconvexity (see Definition \ref{def more convex notions}). 
The importance of these notions goes beyond the lower semicontinuity of supremal functionals and we review, in the next introductory paragraphs, the context of their introduction in the literature as well as the motivation to our analysis. Having in mind the scope of power-law approximation in the applications (see the list of references in Section 1), we start our discussion focusing on its interplay with the broader notion of $\mathcal A-\infty$ quasiconvexity, $\mathcal A$ denoting a generic differential constraint (e.g. $\mathcal A= \rm div$, in the case of plasticity, or $\mathcal A=(\rm curl,\rm div)$ in the case of micromagnetics, or $\mathcal A= \rm curl$, as in our subsequent analysis).

 At this point it is worth to recall that the theory of $\mathcal A$-quasiconvexity has been introduced by Dacorogna, (see e.g. \cite[pages 100-102]{D1}), 
the theory was later formalized in \cite{FMAq}, in the case of constant rank operators, (to which we refer for a detailed treatment of the subject).
It has been then extended to the context of $L^\infty$ problems, first in the case when $\mathcal A={\rm div}$ treated by Bocea and Nesi \cite{BN} and later, with much wider generality, by Ansini and Prinari in \cite{APSIMA, APESAIM}, giving particular emphasis to power-law approximation.

Indeed, departing from the material science's results already mentioned in the introduction, where it was satisfactory to provide sufficient conditions on a supremand $f:\Omega \times \mathbb R^{m}\to [0,+\infty)$, in order to guarantee the variational convergence, as $p\to +\infty$, of functionals of the type
\begin{equation}\label{powerlawfunct}\left(\int_\Omega f^p(x, v(x))\,dx\right)^{\frac{1}{p}}  
\end{equation}
towards
\begin{equation}\label{suppl}
\operatorname*{ess\,sup}_{x \in \Omega} f(x, v(x)),
\end{equation}
with $v(x)$ possibly satisfying $\mathcal A v=0$ (cf. \cite{GNP, BN, BM, CDP, EP} among a wider literature), the asymptotic behaviour of functionals of the type \eqref{powerlawfunct} has been object of investigation, leading to limiting $L^\infty$ energies different from \eqref{suppl}, see for instance \cite{APESAIM, APSIMA, BGP, BPZ, PZ}.


In particular, in \cite[Theorem 4.2]{APSIMA}, it has been computed the $\Gamma$-limit with respect to the $L^\infty$-weak* convergence of \eqref{powerlawfunct}, for Carath\'eodory integrands, under a generic differential constraint $\mathcal A$ on the fields $v$ and
a linear coercivity condition on $f$ on the second variable, i.e.  there exists $\alpha >0$ such that
\begin{equation}\label{coerci}
	f(x,\xi) \geq \alpha |\xi|,
\end{equation}
 for every $\xi \in \mathbb R^{m}$ and a.e. $x \in \Omega$.
Having in mind the case of $\mathcal A= \rm curl$ and $m=N\times n$,
the obtained limit energy has the form \begin{equation}\label{supplimit}
\operatorname*{ess\,sup}_{x \in \Omega} Q_\infty f(x, v(x)),
\end{equation}
where the density $Q_\infty f$ is the so-called $\rm curl-\infty$ quasiconvex envelope of $f(x,\cdot)$, namely the greatest $\rm curl-\infty$ quasiconvex minorant of $f(x,\cdot)$ (see \cite[Section 3.2]{APSIMA} and \cite{APESAIM} for definitions and proofs of this result in a more general framework and \eqref{Qinfty} below for an equivalent definition). 
More precisely, in \cite[Theorem 4.4]{APSIMA} it has been proven that the ${\rm curl}-\infty$ quasiconvexity of $f$ is necessary and sufficient for the $L^p$ approximation of \eqref{suppl} in terms of \eqref{powerlawfunct} in the continuous, homogeneous (and $\rm curl$-free, among more general operators $\mathcal A$) setting, assuming  \eqref{coerci}. 

With ${\rm curl}-\infty$ quasiconvexity playing a crucial role for the attainment of a variational limit with supremal form, it arises naturally the question of comparing this notion with the other (necessary and) sufficient conditions for this variational convergence and consequently with the necessary and sufficient conditions for lower semicontinuity of supremal functionals expressed in Section 3.
It is worth, indeed, to recall that, from the theoretical stand-point, the variational power-law approximation (for instance obtained via $\Gamma$-convergence), guarantees that the limit functional is weakly* lower semicontinuous (see \cite{DM}).  This entails that the  `convexity' hypotheses which provide power-law approximation are sufficient conditions for the lower semicontinuity of the limiting supremal functional, anyway leaving open the necessity condition. With the aim of adopting power-law approximation to get lower semicontinuity of their limiting supremal functionals, the definition of {\it generalized Jensen's inequality} has been introduced in \cite{CDP},  later revisited by \cite{APESAIM} (for general operators $\mathcal A$), leading to the notion of $\rm curl$-Young quasiconvexity, and $\rm curl_{(p>1)}$-Young quasiconvexity (see Definition \ref{def more convex notions}). 
Hence it emerges also the question of comparing and establishing a hierarchy among the notions of ${\rm curl}_{(p>1)}$-Young, ${\rm curl}-\infty$ and ${\rm curl}$-Young quasiconvexity and those introduced in Section 3. At this point it should be emphasized that the question was already completely solved in the scalar case $n=1$ in \cite{BL}, using three different approaches. Performing an $L^p$ approximation, by means of the duality theory in Convex Analysis, and making use of Young measures, level convex envelopes appear as densities of relaxed functionals. 
It is not yet known how to obtain the relaxation in the vectorial framework, and, at the same time, how to deal with the three approaches  mentioned above. 

In order to introduce some of the mentioned properties we will need the concept of (gradient) Young measures. We start recalling the fundamental theorem of Young measure theory, which we present as in \cite[Theorem 4.1]{Rindler}.

Denote by $\mathcal Pr(\mathbb R^{m})$ the set of Borel probability measures defined in $\mathbb R^{m}$.

\begin{theorem}\label{fundamental-theorem} Let $\Omega\subseteq\mathbb{R}^n$ be an open, 
	bounded, connected set with Lipschitz boundary.
	Let $(V_j)_{j\in\mathbb{N}}$ be a sequence bounded in $ L^p(\Omega;\mathbb R^m)$, 
	where $p \in [1,+\infty]$. Then, there exists a subsequence (not explicitly
	labeled) and a family of probability measures,
	$\{\nu_x\}_{x \in \Omega}\subset {\mathcal Pr}(\mathbb R^m)$,
	called the $(L^p)$-Young measure 
generated by the (sub)sequence $(V_j)_{j\in\mathbb{N}}$, 
such that the
	following assertions are true:
	\begin{itemize}
		\item[(i)] The family $\{\nu_x\}_{x \in \Omega}$ is weakly* measurable, that is, for all Carath\'eodory integrands
		$f : \Omega \times \mathbb R^m \to \mathbb R$, the compound function
		$x\mapsto  <f (x, \cdot), \nu_x>=\int_{\mathbb R^m} f(x, \xi)\,d \nu_x(\xi), \; x \in \Omega$ is Lebesgue measurable,
		\item[(ii)]
		If $p \in [1,+\infty)$, it holds that
		$\int_\Omega \int_{\mathbb R^{m}}|\xi|^p\,d\nu_x(\xi)\,dx <+\infty,$
		or, if $p =\infty$, there exists a compact set $K \subset \mathbb R^m$ such that
		${\rm supp}\,\nu_x \subset K$ for a.e. $x \in \Omega$.
		\item[(iii)] For all Carath\'eodory integrands $f : \Omega \times \mathbb R^m \to \mathbb R$ with the property that the
		family $(f (x, V_j ))_{j\in\mathbb{N}}$ is uniformly bounded in $L^1$ and equiintegrable, it holds that
		$f(x,V_j)\rightharpoonup \left( x \mapsto \int_{\mathbb R^m}f(x,\xi)\,d\nu_x(\xi)\right)$ in $L^1$.
	\end{itemize}
\end{theorem}

In the case $p<\infty$, $(iii)$ follows by $(i)$ and $(ii)$, (cf. \cite[Problem 4.3]{Rindler}) and thus we will refer to
($L^p$)-Young measure to any family of parametrized measures $\nu=\{\nu_x\}_{x\in \Omega}$ satisfying $(i)$ and $(ii)$, either if $p \in [1,+\infty)$ or $p=+\infty$.

We will write $V_j\overset{Y}{\to}\nu$ to refer to the sequence $(V_j)_{j\in\mathbb{N}}$ which generates the Young measure $\nu$. 
The Young measure $\nu$ is said to
be homogeneous if there is a measure $\nu_0 \in \mathcal Pr(\mathbb R^m)$  such that $\nu_x = \nu_0$ for $\mathcal L^n$- a.e. $x \in  \Omega.$

In the sequel, we will be interested in Young measures that are generated by sequences of gradients.  Recall that, given an $(L^p)$-Young measure $\nu \equiv\{\nu_x\}_{x\in \Omega}$ 
we say that $\nu$ is a $W^{1,p}$-{\it gradient Young measure}, $p \in [1,+\infty]$, 
if there exists $u_j \in W^{1,p}(\Omega)$ such that $\nabla u_j$ generates $\nu\equiv\{\nu_x\}_{x \in \Omega}$. If $p=+\infty$ we can simply say that $\{\nu_x\}_{x\in \Omega}$ is a  gradient Young measure (we refer to  Kinderlehrer and Pedregal \cite{KP1991, KP1994}, and to Rindler \cite{Rindler}, where these measures are called $W^{1,\infty}$-gradient Young measures).
A {\it homogeneous} $W^{1,p}$-{\it gradient Young measure} ({\it homogeneous gradient Young measure} respectively) is a $W^{1,p}$-{\it gradient Young measure} (a {\it gradient Young measure} respectively) which is homogeneous in the above mentioned sense. 


Having in mind the more general setting of $\mathcal A$- free fields in $L^p$, and in order to understand the results available in literature dealing with $L^p$-approximation, the Fundamental Theorem can be considered also to justify the introduction of $\mathcal A-\infty$ Young measures,  as in \cite[Section 2]{FMAq} and \cite{APESAIM}. Indeed, without loss of generality, these are measures generated by sequences in $L^\infty(\Omega; \mathbb R^{d\times N}) \cap {\rm Ker} \mathcal A$ (where ${\rm Ker}\mathcal A$ denotes the kernel of the operator $\mathcal A$) uniformly bounded in the $L^\infty$ norm (or  equivalently, possibly passing to a subsequence, weakly*
converging in $L^\infty$ (see \cite[Section 2]{FMAq})). 
In the case $\mathcal A= {\rm curl}$, it results that
${\rm curl} -\infty$ Young measures are gradient Young measures. In the following, we will adopt the latter terminology.


\begin{definition}\label{def more convex notions}
Let $f:\mathbb{R}^{N\times n}\longrightarrow \mathbb{R}$ be a Borel measurable function.
\begin{enumerate}
\item Assume that $f$ is lower semicontinuous and bounded from below. We say that $f$ is ${\rm curl}_{(p>1)}$-Young quasiconvex, if 
\begin{align}\label{curlp}
f\left(\int_{\mathbb R^{N\times n}}\xi\, d\nu_x(\xi) \right)\leq \operatorname*{ess\,sup}_{y \in Q}  \left(\mathop{\nu_y - \rm ess\,sup}_{\xi \in \mathbb R^{N\times n}}f(\xi)\right),\ \mathrm{a.e.}\ x\in Q
\end{align}
whenever $\nu\equiv\{\nu_x\}_{x \in Q}$ is a $W^{1,p}$-gradient Young measure for every $p\in(1,\infty)$.

\item Assume that $f$ is lower semicontinuous and bounded from below. We say that $f$ is ${\rm curl}$-Young quasiconvex, if 
\begin{align}
\label{curlY}
f\left(\int_{\mathbb R^{N\times n}}\xi\, d\nu_x(\xi) \right)\leq \operatorname*{ess\,sup}_{y \in Q}  \left(\mathop{\nu_y - \rm ess\,sup}_{\xi \in \mathbb R^{N\times n}}f(\xi)\right),\ \mathrm{a.e.}\ x\in Q,
\end{align}
whenever $\nu\equiv\{\nu_x\}_{x \in Q}$ is a 
gradient Young measure.

\item Assume that $f$ is non-negative. We say that $f$ is $\mathbf{{\rm curl}-\infty}$ quasiconvex if for every $\xi \in \mathbb R^{N\times n}$
\begin{align}\label{fAinfty}
f(\xi) = \lim_{p\to +\infty} \inf\left\{\left(\int_Q f^p(\xi + D u (x))\,dx\right)^{\tfrac{1}{p}}: u \in W^{1,\infty}_{\rm per} (Q; \mathbb R^N)\right\}.
\end{align}
	
\end{enumerate}	
\end{definition}

\begin{remark}\label{remark on Young quasiconvexity defs} 
\begin{enumerate}			
\item[(i)] We observe that the double essential suprema in (1) and (2) of the previous definition are meaningful because the function $f$ is assumed to be lower semicontinuous and bounded from below. In principle, one can give definitions of ${\rm curl}_{(p>1)}$-Young and ${\rm curl}$-Young quasiconvexity without these assumptions, adding the condition that the right-hand sides of (1) and (2) make sense. This may be the approach done in \cite{APESAIM,CDP}. Next, we describe why the assumptions of lower semicontinuity and boundedness from below are sufficient for this goal. We first observe that there is no loss of generality assuming the bound from below is zero. In that case, $$\mathop{\nu_y - \rm ess\,sup}_{\xi \in \mathbb R^{N\times n}}f(\xi)=\lim_{k\to \infty}||f||_{L^k(\mathbb{R}^{N\times n};\nu_y)}.$$ This identity ensures the Lebesgue measurability of $y\mapsto\mathop{\nu_y - \rm ess\,sup}_{\xi \in \mathbb R^{N\times n}}f(\xi)$ in view of the Lebesgue measurability of $$y\mapsto\int_{\mathbb{R}^{N\times n}}|f(\xi)|^k\,d\nu_y$$ which follows from Theorem \ref{fundamental-theorem} (i), extended to normal integrands by using \cite[Corollary 6.30]{Fonseca-Leoni-Lp}. 

\item[(ii)]  The notion of ${\rm curl}_{(p>1)}$-Young quasiconvexity already appeared in \cite[eq. (3.1) in Theorem 3.1, under the name of 'generalized Jensen's inequality']{CDP}. There, the function $f$ is considered with also $x$ and $u$ dependence.

\item[(iii)] Note that, if $1\le p<q\le\infty$, then every $W^{1,q}$-gradient Young measure (gradient Young measure if $q=\infty$) is also a $W^{1,p}$-gradient Young measure. Therefore, to verify ${\rm curl}_{(p>1)}-$Young quasiconvexity it is enough to check \eqref{curlp} for parametrized measures that are $W^{1,p}$-gradient Young measures for every $p\in (p_0,\infty)$ with $p_0>1$.

\item[(iv)] The set $Q$ in the definition of ${\rm curl}$-Young quasiconvexity can be replaced by any other bounded open set as observed in \cite[Remark 4.3]{APESAIM} in the context of $\mathcal A -$quasiconvexity under a coercivity assumption.

\item[(v)] the notion of ${\rm curl}-\infty$ quasiconvexity can be found in \cite[Definition 3.3]{APSIMA}, with $\mathcal A= {\rm curl}$.
\end{enumerate}		
\end{remark}

Next we provide some characterizations of ${\rm curl}_{(p>1)}$-Young  quasiconvexity and of ${\rm curl}$-Young quasiconvexity.

\begin{proposition}\label{curlpYoungCharacterizations}
Let $f:\mathbb{R}^{N\times n}\longrightarrow \mathbb{R}$ be a lower semicontinuous function and bounded from below. Then the following conditions are equivalent
\begin{enumerate}
\item[(i)] $f$ is ${\rm curl}_{(p>1)}$-Young quasiconvex;
\item[(ii)] $f$ verifies 
\begin{align}
f\left(\int_{\mathbb R^{N \times n}}\xi\,d\nu_x(\xi)\right) \leq \mathop{\nu_x - \rm ess\,sup}_{\xi \in \mathbb R^{N\times n}} f(\xi)\
\hbox{ for a.e. } x \in Q
\end{align}
whenever $\nu\equiv\{\nu_x\}_{x \in Q}$ is a $W^{1,p}$-gradient Young measure for every $p\in(1,\infty)$;
\item[(iii)] $f$ verifies
\begin{align}\label{fcurlYoungh}
f\left(\int_{\mathbb R^{N \times n}}\xi\,d\nu(\xi)\right) \leq  \mathop{\nu - \rm ess\,sup}_{\xi \in \mathbb R^{N\times n}} f(\xi) 
\end{align} 
whenever $\nu$ is a homogeneous $W^{1,p}$-gradient Young measure for every $p\in(1,\infty)$.
\end{enumerate}
Moreover, in the definition of ${\rm curl}_{(p>1)}$-Young quasiconvexity the domain $Q$ can be replaced by any open, bounded, connected set $\Omega\subseteq\mathbb{R}^n$, with Lipschitz boundary.  

An analogous statement holds by replacing ${\rm curl}_{(p>1)}$-Young quasiconvexity by ${\rm curl}$-Young quasiconvexity and $W^{1,p}$-gradient Young measures for every $p\in(1,\infty)$ by gradient Young measures.
\end{proposition}

\begin{remark}

 We observe that the result just stated is still true without the lower semicontinuity and boundedness from below assumptions, if conditions \eqref{curlp} and \eqref{curlY} are meaningful in the sense of Remark \ref{remark on Young quasiconvexity defs} (i).
\end{remark}

\begin{proof}[Proof]
Clearly (ii) entails (i). On the other hand, (i) is equivalent to (iii). Indeed, (i) implies (iii) because homogeneous $W^{1,p}$-gradient Young measures are particular cases of $W^{1,p}$-gradient Young measures. In turn, (iii) implies (i) taking into account that given an arbitrary $W^{1,p}$-gradient Young measure  $\{\nu_x\}_{x \in Q}$, each $\nu_x$ (for almost every $x$ fixed) is a homogeneous $W^{1,p}$-gradient Young measure ({\sl cf.} \cite[Proposition 5.14 and Remark 5.15]{Rindler}). Similarly, (i) implies (ii).

In particular, by (iii), we conclude that (i) does not depend on the domain $Q$. Actually, reasoning as above we can show that (iii) is equivalent to
$$f\left(\int_{\mathbb R^{N\times n}}\xi\, d\nu_x(\xi) \right)\leq \operatorname*{ess\,sup}_{y \in \Omega}  \left(\mathop{\nu_y - \rm ess\,sup}_{\xi \in \mathbb R^{N\times n}}f(\xi)\right),\ \mathrm{a.e.}\ x\in \Omega$$
whenever $\nu\equiv\{\nu_x\}_{x \in \Omega}$ is a $W^{1,p}$-gradient Young measure for every $p\in(1,\infty)$
where $\Omega\subseteq\mathbb{R}^n$ is any open, bounded, connected set with Lipschitz boundary.

The proof of the case of ${\rm curl}$-Young quasiconvex functions is analogous to the previous one.
\end{proof}

Next we provide a characterization of ${\rm curl}-\infty$ quasiconvexity through a power-law approximation of quasiconvex hulls.
 
\begin{proposition}\label{curlinftyqcxlsc}
Let $f:\mathbb{R}^{N\times n}\longrightarrow [0,\infty)$ be a Borel measurable function. Then $f$ is $\rm curl-\infty$ quasiconvex if and only if 
\begin{align}\label{cahrQfinfty}
f(\xi)= \lim_{p\to +\infty} (\mathcal Q (f^p))^{1/p}(\xi),
\end{align}
where $\mathcal Q (f^p)$ stands for the quasiconvex envelope of $f^p$. In particular, if $f$ is $\rm curl-\infty$ quasiconvex then $f$ is lower semicontinuous and locally bounded.
\end{proposition}

\begin{remark}
Note that, combining \eqref{cahrQfinfty} with the several characterizations of a quasiconvex envelope provided in Lemmas \ref{lemmaenvelope} and \ref{lemmaenvelope2} (this second lemma applies to functions with linear growth), we can explicit $\rm curl-\infty$ quasiconvex functions in terms of a limit of several type of minimization problems.
\end{remark}

\begin{proof}[Proof] The characterization of $\rm curl-\infty$ quasiconvexity follows from Lemma $\ref{lemmaenvelope}$ applied to $f^p$.

Now assume that $f$ is a $\rm curl-\infty$ quasiconvex function. By Remark \ref{rem-qcxdef} (iii), $\mathcal Q (f^p)$ is lower semicontinuous. Therefore, also $(\mathcal Q (f^p))^{1/p}$ is lower semicontinuous and thus, $f$ being the limit of a monotone increasing family of lower semicontinuous functions, it is also lower semicontinuous. Finally, to show that $f$ is locally bounded, we invoke \cite[Exercise 5.11]{AFP}. Indeed, since $\mathcal Q (f^p)$ is quasiconvex, it is also separately convex and we get that, for $h\in\mathbb{N}$,
\begin{align*}
\sup_{\xi\in[-h,h]^{N\times n}} \mathcal Q (f^p)(\xi) &\le (2^{N\times n+1}-1)\max\left\{ Q (f^p)(\zeta):\ \zeta\in[-h-1,h+1]^{N\times n}\cap\mathbb{Z}^{N\times n}\right\}\\
& \le  (2^{N\times n+1}-1)\max\left\{f^p(\zeta):\ \zeta\in[-h-1,h+1]^{N\times n}\cap\mathbb{Z}^{N\times n}\right\}.
\end{align*}
implying 
\begin{align*}
\sup_{\xi\in[-h,h]^{N\times n}} \mathcal (Q (f^p))^{1/p}(\xi) &\le (2^{N\times n+1}-1)^{1/p}\max\left\{ f(\zeta):\ \zeta\in[-h-1,h+1]^{N\times n}\cap\mathbb{Z}^{N\times n}\right\}
\end{align*}
that provides a local bound for $f$ if we have in mind \eqref{cahrQfinfty}.
\end{proof}
	
As a complement to Theorem \ref{thm relations between convexity notions}, we state a result, establishing some relations between the convexity notions introduced above and those of the previous section. We will consider the following growth and coercivity conditions

\begin{align*}
(G) \qquad \exists\ C>0:\ f(\xi) \leq C(1+|\xi|)\; \hbox{ for every }\xi \in \mathbb R^{N\times n}
\end{align*}

\begin{align*}
(C) \qquad \exists\ \alpha, \beta>0:\ f(\xi)\geq \alpha |\xi|-\beta\; \hbox{ for every }\xi \in \mathbb R^{N\times n}.
\end{align*}

\begin{theorem}\label{Other} Let $f:\mathbb R^{N\times n}\to \mathbb R$ be a Borel measurable function, bounded from below.
\begin{enumerate} 
\item If $f$ is level convex then it satisfies inequality \eqref{fcurlYoungh} for every $W^{1,p}$-gradient Young measure, for $p \in (1,+\infty]$. In particular, if it is also lower semicontinuous then $f$ is ${\rm curl}_{(p>1)}-$Young quasiconvex and ${\rm curl}$-Young quasiconvex. 

\item If $f:\mathbb R^{N\times n}\to [0,+\infty)$ is level convex, lower semicontinuous, and satisfies the coercivity condition (C), then $f$ is $\rm curl-\infty$ quasiconvex.

\item\label{1a} If $f: \mathbb R^{N\times n}\to[0, +\infty)$ and polyquasiconvex with $f=g\circ T$, for $g:\mathbb{R}^{\tau(n,N)}\longrightarrow\mathbb{R}$ level convex and lower semicontinuous and satisfying $(C)$ in $\mathbb{R}^{\tau(n,N)}$ then $f$ is ${\rm curl}-\infty$ quasiconvex. 
\item If $f$ is polyquasiconvex with $f=g\circ T$, for $g:\mathbb{R}^{\tau(n,N)}\longrightarrow\mathbb{R}$ level convex then $f$ satisfies \eqref{fcurlYoungh} for every $W^{1,p}$-gradient Young measure
for $p \in (1,+\infty]$. In particular, if $f$ is also lower semicontinuous then $f$ is ${\rm curl}_{(p>1)}-$Young quasiconvex and ${\rm curl}$-Young quasiconvex. 

\item If $f$ is ${\rm curl}_{(p>1)}-$Young quasiconvex, then $f$ is  ${\rm curl}-$Young quasiconvex.

\item Let $f:\mathbb R^{N\times n}\to [0,+\infty)$ be  a ${\rm curl}_{(p>1)}-$Young quasiconvex function satisfying the coercivity condition (C). If, moreover, either $f$ is upper semi-continuous or $f$ satisfies the growth condition (G), then $f$ is $\rm curl-\infty$ quasiconvex.

\item If $f:\mathbb R^{N\times n}\to [0,+\infty)$ is ${\rm curl}-\infty$ quasiconvex then $f$ is ${\rm curl}-$Young quasiconvex.  Moreover if it also satisfies the growth condition (G), then $f$ is ${\rm curl}_{(p>1)}-$Young quasiconvex.

\item If $f:\mathbb R^{N\times n}\to [0,+\infty)$ is $\rm curl-$Young quasiconvex and locally bounded and satisfies the coercivity condition (C), then $f$ is also $\rm curl-\infty$ quasiconvex.
 Moreover, if $f$ satisfies the growth condition $(G)$, then $f$ is $\rm curl_{(p>1)}$- Young quasiconvex.

\item If $f$ is either ${\rm curl}$-Young quasiconvex or ${\rm curl}_{(p>1)}-$Young quasiconvex then $f$ is strong Morrey quasiconvex.
Moreover, if we only assume \eqref{fcurlYoungh} for every $W^{1,p}$-gradient Young measure with $p>1$ or for any gradient Young measure, then $f$ is weak Morrey quasiconvex.



Also, if $f$ satisfies \eqref{fcurlYoungh} for every gradient Young measure, or for every $W^{1,p}$-gradient Young measure for every $p>1$ then $f$ is rank-one quasiconvex. 


 \noindent 

\item If $f:\mathbb R^{N\times n}\to [0,+\infty)$ is $\rm curl-\infty$ quasiconvex, then $f$ is strong Morrey quasiconvex.
In particular, it is periodic-weak Morrey quasiconvex in any cube $C\subset \mathbb R^n$, weak Morrey quasiconvex, and  rank-one quasiconvex.
\end{enumerate}
\end{theorem}

The following picture depicts the implications stated in Theorem \ref{Other} concerning the notions introduced in this section. Note that the theorem contains broader results.

\begin{figure}[H]
\begin{tikzcd}[arrows=Rightarrow]
\ & {f \text{ level convex}} 
\arrow[out=-45,in=45]{d}{}[right]{{\rm Thm\,} \ref{thm relations between convexity notions}}
\arrow[ddl, "(1)^*", to path={[swap, pos=0.75]..controls +(-2.5,0) and  +(0,1.5).. (\tikztotarget) \tikztonodes}]
\arrow[ddr, "(2)^*_+", to path={[pos=0.75]..controls +(2.5,0) and  +(0,1.5).. (\tikztotarget) \tikztonodes}]
\arrow[ddd, "(1)^*", to path={[pos=0.2]..controls +(-2.5,0) and  +(-2.7,1.5).. (\tikztotarget) \tikztonodes}]
& \ 
\\
\ & {f \text{ polyquasiconvex}} 
\arrow[dr, "(3)^*_+", to path={[pos=0.25]..controls +(2.5,0) and  +(0,1.5).. (\tikztotarget) \tikztonodes}]
\arrow[dd, "(4)", to path={[pos=0.45]..controls +(2.5,0) and  +(2.7,1.5).. (\tikztotarget) \tikztonodes}]
\arrow[dl, "(4)", to path={[pos=0.55]..controls +(-2.5,0) and  +(0,1.5).. (\tikztotarget) \tikztonodes}]
& \ 
\\
\begin{array}{c}
f\ \mathrm{curl}_{(p>1)}\text{-Young}\\ \text{quasiconvex}
\end{array}
\arrow[rr, "(6)^*_+", to path={[pos=0.5]..controls +(3.5,0.2) and  +(-3.5,0.2).. (\tikztotarget) \tikztonodes}]
\arrow[dr, "(5)", to path={[pos=0.5]..controls +(1,-1.5) and  +(-1.2,0.2).. (\tikztotarget) \tikztonodes}]
\arrow[ddr, "(9)", to path={[swap,pos=0.65]..controls +(-2.5,0) and  +(-2.5,0).. (\tikztotarget) \tikztonodes}]
\arrow[dddr, "(9)", to path={[swap,pos=0.65]..controls +(-2.5,0) and  +(-2.5,0).. (\tikztotarget) \tikztonodes}]
\arrow[ddddr, "(9)", to path={[swap,pos=0.65]..controls +(-2.5,0) and  +(-2.5,0).. (\tikztotarget) \tikztonodes}]
& \  & \begin{array}{c}f\ \mathrm{curl}-\infty\\ \text{quasiconvex}\end{array}
\arrow[ll, "(7)^*", to path={[pos=0.5]..controls +(-3.5,-0.2) and  +(3.5,-0.2).. (\tikztotarget) \tikztonodes}]
\arrow[dl, "(7)", to path={[swap,pos=0.5]..controls +(-1,-1.5) and  +(1.2,0.2).. (\tikztotarget) \tikztonodes}]
\arrow[ddl, "(10)", to path={[pos=0.8]..controls +(2.5,0) and  +(2.5,0).. (\tikztotarget) \tikztonodes}]
\arrow[dddl, "(10)", to path={[pos=0.8]..controls +(2.5,0) and  +(2.5,0).. (\tikztotarget) \tikztonodes}]
\arrow[ddddl, "(10)", to path={[pos=0.8]..controls +(2.5,0) and  +(2.5,0).. (\tikztotarget) \tikztonodes}]
\\
\ & \begin{array}{c}f\ \mathrm{curl}\text{-Young}\\ \text{quasiconvex}\end{array}
\arrow[ur, "(8)^*_+", to path={[swap,pos=0.25]..controls +(2.5,0) and  +(0,-1.5).. (\tikztotarget) \tikztonodes}]
\arrow[ul, "(8)^*_+", to path={[pos=0.3]..controls +(-2.5,0) and  +(0,-1.5).. (\tikztotarget) \tikztonodes}]
\arrow[out=-135,in=135]{d}{}[right]{(9)}
\arrow[dd, "(9)", to path={[pos=0.75]..controls +(-2.5,0) and  +(-3.5,0).. (\tikztotarget) \tikztonodes}]
\arrow[ddd, "(9)", to path={[pos=0.75]..controls +(-2.5,0) and  +(-3.5,0).. (\tikztotarget) \tikztonodes}]
 & \ 
\\ 
\ & \ {f \text{ strong Morrey quasiconvex}} \arrow[out=-45,in=45]{d}{}[right]{{\rm Thm\,} \ref{thm relations between convexity notions}} & \ 
\\
\ & {f \text{ weak Morrey quasiconvex}} \arrow[out=-45,in=45]{d}{}[right]{{\rm Thm\,} \ref{thm relations between convexity notions}^*} &\ 
\\
\ & {f \text{ rank one quasiconvex}} & \ 
\end{tikzcd}
\caption{The figure shall be considered accompanied by Theorem \ref{Other}. Namely, the labels to each implications refer to the items in the theorem containing the referred implication and, when additional hypotheses are required, the label is signed with $^*$. The subscript + means that the corresponding implication only applies to nonnegative functions.}
\end{figure}

\begin{remark}\label{Counteremark}
\begin{enumerate}
\item[(i)] In view of (5), (1) improves \cite[Proposition 3.3 (2)]{APESAIM} in the case of $\mathcal A=\rm curl$. Observe also that the proof of (1) ensures that level convex functions satisfy condition \eqref{fcurlYoungh} without any assumption of lower semicontinuity nor on the sign of the function. These assumptions are only needed to relate with ${\rm curl}_{(p>1)}-$Young quasiconvexity.
\item[(ii)] As a by-product of the proof of (2), we have obtained a generalization of \cite[Proposition 5.1]{PZ} asserting 
$$\lim_{p\to +\infty}\left((f^p)^{**}\right)^{1/p}=f$$ under the weaker coercivity assumption $(C)$. (The notation $(f^p)^{**}$ stands, as usual, for the bidual of $f^p$, cf. \cite[Definition 2.41 (ii)]{Dbook}.
\item[(ii)] We observe that \eqref{1a} has been already proven under continuity assumptions of $g$ in \cite[Proposition 5.7]{APSIMA}.
\item[(iv)] Regarding (4), the fact that a polyquasiconvex function is also ${\rm curl}$-Young quasiconvex has been already proven in \cite[Proposition 6.3]{APESAIM} in the case of a nonnegative  and a lower senjcontinuous $g$, since their argument exploits our first implication of (7).
\item[(v)] We observe that the first implication of (7) in the above result has been proved in \cite[Proposition 3.4]{APESAIM}, in the more general setting of $\mathcal A-\infty$-Young measures, under the continuity assumption on $f$. As seen in (7), this extra requirement can be removed in the case $\mathcal A= \rm curl$. Note, however that, in this case, as observed in Remark \ref{curlinftyqcxlsc}, $f$ is lower semicontinuous.
\item[(vi)] The assertion of (9) related to rank-one quasiconvexity has been proved in \cite[Proposition 6.4 (3)]{APESAIM.} under the extra lower semicontinuity assumption. 
\item[(vii)] If $n=1$ or $N=1$, $f$ is level convex if and only if it satisfies \eqref{fcurlYoungh} for every $W^{1,p}$-gradient Young measure, for any $p \in (1,+\infty]$. In particular, if $f$ is also lower semicontinuous it is level convex (equivalently strong Morrey quasiconvex) if and only if it is $\rm curl_{(p>1)}$-Young quasiconvex and if and only if it is $\rm curl$-Young quasiconvex.  
The above considerations follow by (1) and (9), recalling that  in the scalar case, rank-one quasiconvexity reduces to level convexity. Note, however, that, with this restriction, in the proof of (9) we can invoke the zig-zag lemma, cf. \cite[Lemma 20.2]{DM}, in place of the construction of approximate solutions.   We recall that if $N=1$ there exist lower semicontinuous functions which are weak Morrey quasiconvex but neither ${\rm curl}_{(p>1)}$-Young quasiconvex nor ${\rm curl}$-Young quasiconvex.

\item[(viii)] If $n=1$ or $N=1$ and $f:\mathbb R^{N\times n}\to [0,+\infty)$ satisfies $(C)$, then it is level convex, lower semicontinuous (equivalently strong Morrey quasiconvex) if and only if it is $\rm curl-\infty$ quasiconvex. 

\end{enumerate}
\end{remark}

\begin{proof}
\begin{enumerate}
\item The argument to prove this implication follows from the fact that level convexity and Borel measurability entail the supremal Jensen's inequality. Namely
$$
f\left(\int_\Omega \varphi\,d\mu\right)\leq \mu-\operatorname*{ess\,sup}_{x\in\Omega} f(\varphi(x))
$$
for every probability measure $\mu$ on $\mathbb{R}^{d}$ supported on the open set $\Omega \subseteq \mathbb{R}^{d}$, and every $\varphi \in L^1_\mu(\Omega;\mathbb{R}^{n})$ (see Theorem \ref{Jensensupremalscalar} for a proof). 

We apply the previous inequality with $\varphi = id$, $\Omega= \mathbb R^{N\times n}$, and $\mu$ any probability measure in $\mathbb{R}^{N\times n}$. In particular, we observe that $f$ satisfies \eqref{fcurlYoungh}, whenever $\nu$ is a homogeneous $W^{1,p}$-gradient Young measure for every $p\in(1,\infty]$ which, by Proposition \ref{curlpYoungCharacterizations}, implies that $f$ is ${\rm curl}_p-$Young quasiconvex and ${\rm curl}$-Young quasiconvex provided $f$ is also lower semicontinuous.

\item This result has been obtained under the requirement that $f(\xi)\geq \alpha |\xi|$, for every $\xi \in \mathbb R^{N\times n}$ in \cite[Proposition 2.9]{APESAIM}. Here we observe that this condition can be relaxed. We start observing that
\begin{align}
\label{limsupfpaastast}\limsup_{p\to +\infty} (\mathcal Q (f^p)(\xi))^{1/p}\leq f(\xi).
\end{align}
Having in mind that $(f^p)^{**}\le \mathcal Q (f^p)$, we will prove the opposite inequality, showing that $$f(\xi)\le \limsup_{p\to +\infty} ((f^p)^{**}(\xi))^{1/p}.$$

To this end we will invoke \cite[Corollary 3.11]{P0} which relies on the relaxation result \cite[Theorem 3.9]{P0} where the coercivity is only needed to ensure boundedness of gradients. Therefore \cite[Corollary 3.11]{P0} can be generalized to the coercivity condition $f(\xi)\ge \alpha|\xi|-\beta$. However, \cite[Corollary 3.11]{P0} requires continuity and linear growth from above. To deal with these assumptions, we replace $f$ by its Pasch-Hausdorff transform, as in \cite[Proposition 2.9]{APESAIM}. This is defined as
$f_\lambda(\xi):=\inf\{\max\{f(\eta), \lambda |\xi-\eta|\}: \eta \in \mathbb R^{N\times n}\}$ for every $\lambda >0$ and it turns out that $f_\lambda$ is level convex, continuous and $f=\sup_{\lambda >0} f_\lambda$. Moreover $f_\lambda(\xi)\leq \max\{f(0), \lambda |\xi|\} \leq f(0)+ \lambda |\xi|$, so it has linear growth from above.
For what concerns the coercivity condition, it results that 
\begin{align*}
	f_\lambda(\xi)& \geq \inf\left\{\max\{\alpha|\eta|-\beta, \lambda |\xi-\eta|-\beta\}: \eta \in \mathbb R^{N\times n}\right\}\\
	& =\inf\left\{\max\{\alpha|\eta|,\lambda |\eta-\xi|\}: \eta \in \mathbb R^{N\times n}\right\}-\beta\\
	& \geq\inf\left\{\max \{\alpha |\eta|,\alpha |\eta-\xi|\}: \eta \in \mathbb R^{N\times n}\right\}- \beta; \hbox{ for } \lambda \geq \alpha.
\end{align*}
Then, in view of the level convexity of $|\cdot|$,
\begin{align*}
		f_\lambda(\xi)\geq \frac{1}{2}\alpha |\xi|-\beta, \hbox{ for } \lambda \geq \alpha.
\end{align*}

Applying \cite[Corollary 3.11]{P0} to $f_\lambda$, we have the same chain of inequalities as in \cite[Proposition 2.9]{APESAIM}, i.e. 
\begin{align*}
	f(\xi)&=\sup_{\lambda>0}f_\lambda(\xi)=\sup_{\lambda >0}(\sup_{p>1}(((f_\lambda)^p)^{**}(\xi))^{1/p}&\\
	&=\sup_{p>1}\sup_{\lambda >0}(((f_\lambda)^p)^{**}(\xi))^{1/p}\leq \sup_{p>1}((f^p)^{**}(\xi))^{1/p}.
\end{align*}
Finally, observe that, since $\{((f^p)^{**}(\xi))^{1/p}\}$ is nondecreasing, $$\sup_{p>1}((f^p)^{**}(\xi))^{1/p}=\lim_{p\to +\infty} ((f^p)^{**}(\xi))^{1/p}$$ concluding the proof.

 \item The proof develops along the lines of \cite[Proposition 5.7]{APSIMA}. Taking into account the arguments developed in (2), we can deal with the weaker assumptions of our statement, allowing us to write $$\lim_{p\to +\infty}\left((g^p)^{**}\right)^{1/p}(T(\xi))= g(T(\xi))= f(\xi).$$
On the other hand,
\begin{equation}\label{1ineq}
	f(\xi) = g(T (\xi)) = \lim_{p\to +\infty}
	((g^p)^
	{**})^{1/p}(T (\xi)) \leq \lim_{p\to +\infty}
	f_p(\xi) \leq f(\xi),\end{equation}
where, adopting the same notation as in \cite{APESAIM} and \cite{APSIMA},
$$
f_p(\xi):=\inf\left\{\left(\int_Q f^p(\xi + D u (x))\,dx\right)^{\tfrac{1}{p}}: u \in W^{1,\infty}_{\rm per} (Q; \mathbb R^N)\right\}.
$$
Indeed the polyconvex function $(g^p)^{\ast \ast}(T(\xi))\leq g^p(T(\xi))=f^p(\xi)$, from which the first inequality in \eqref{1ineq} follows.
Finally, \eqref{1ineq} concludes the proof of the statement. 

\item If follows as the proof of (1), by applying Jensen's inequality with $W^{1,p}$-gradient Young measures, ($p \in [1,\infty]$) to the function $g$, recalling that $T$ is quasiaffine and invoking \cite[Corollary 5.12]{Rindler}.

\item As observed in Remark \ref{remark on Young quasiconvexity defs}, gradient Young measures are $W^{1,p}$-gradient Young measures for every $1\le p\le\infty$. This entails the desired implication.

\item In the case $f$ is upper semicontinuous, the result follows from \cite[Theorem 3.1]{CDP} together with \cite[Theorem 4.4]{APSIMA} (note that these results are still valid under the current weaker coercivity assumption). We stress that Theorem 4.4 in \cite{APSIMA} requires the continuity of the function.

Regarding the other case, first we invoke \cite[Theorem 2.2, (61), and Remarks 3.3 and 5.1]{PZ}, where it has been proven that the $\Gamma$-limit with respect to the uniform ($L^\infty$) convergence in $C(\overline Q;\mathbb R^N)$ of 
\begin{align}\label{Ip}
F_p(u):=\left\{\begin{array}{ll}
\displaystyle \left(\int_Q f^p(D u(x))\,dx\right)^\frac{1}{p} &\hbox{ if }u \in W^{1,p}(Q;\mathbb R^N)\cap C(\overline Q;\mathbb R^N),\vspace{0.2cm}\\ 
+ \infty &\hbox{ otherwise}
\end{array}\right.
\end{align}   
is given by 
\begin{align}\label{Fgammalimit}
F(u)=\left\{\begin{array}{ll}\displaystyle \operatorname*{ess\,sup}_{x \in Q} \mathcal{Q}_\infty f(D u) & \hbox{ if }u \in W^{1,\infty}(Q;\mathbb R^N)\cap C(\overline Q;\mathbb R^N),\vspace{0.2cm}\\
+ \infty &\hbox{ otherwise}
\end{array}\right.
\end{align}
where  
\begin{align}\label{Qinfty}
\mathcal{Q}_\infty f(\xi):=\sup_{n \in\mathbb{N}} (\mathcal{Q}(f^n)(\xi))^{1/n}=\lim_{p\to \infty}(\mathcal{Q}(f^p)(\xi))^{1/p}.
\end{align}
	
On the other hand, since $f$ is $\rm curl_{(p>1)}$-Young quasiconvex,  by \cite[Theorem 3.1]{CDP}, the $\Gamma$-limit (with respect to the uniform convergence in $C(\overline Q;\mathbb R^N)$) of the functionals $F_p$ coincides with $\operatorname*{ess\,sup}_{x\in Q}f(D u(x))$, when $u\in W^{1,\infty}(Q;\mathbb R^N)\cap C(\overline Q;\mathbb R^N)$. Therefore we can conclude that 
$$\operatorname*{ess\,sup}_{x \in Q} f(D  u(x))= \operatorname*{ess\,sup}_{x \in Q} \mathcal Q_\infty f(D u(x)),$$
for every $u \in W^{1,\infty}(Q;\mathbb R^N)\cap C(\overline Q;\mathbb R^N)$. Applying the equality to linear fuctions $u$ with $D u=\xi$ arbitrary, we obtain
$$f(\xi)= \mathcal Q_\infty f(\xi).$$
Having in mind the characterization of $\rm curl-\infty$ quasiconvexity  provided by \eqref{cahrQfinfty} and recalling \eqref{Qinfty}, this proves our claim. 

\item We prove the second part of the statement, the first being very similar, just observing that \cite[Theorem 7.15]{Rindler} applies to gradient Young measures with no need of any growth condition.  Let $\{\nu_x\}_{x\in Q}$ be a parametrized measure that is a $W^{1,p}$-gradient Young measure for every $p\in(1,\infty)$. The growth assumption (G) allows to apply \cite[Theorem 7.15]{Rindler} to get, 
\begin{align}\label{zappale}
\mathcal Q (f^p)\left(\int_{\mathbb R^{N\times n}} \xi\,d \nu_x(\xi)\right)\leq \int_{\mathbb R^{N\times n}} f^p(\xi)\,d \nu_x(\xi).
\end{align}

Therefore $$\mathcal Q (f^p)\left(\int_{\mathbb R^{N\times n}} \xi\,d \nu_x(\xi)\right)\leq  \int_{\mathbb R^{N\times n}} f^p(\xi)\,d \nu_x(\xi)\leq \mathop{\nu_x - \rm ess\,sup}_{\xi \in \mathbb R^{N\times n}} f^p(\xi).$$
Taking the power $\tfrac{1}{p}$ on this inequality and passing to the limit as $p\to \infty$ we get the ${\rm curl}_{(p>1)}-$Young quasiconvexity of $f$, having in mind the assumption that $f$ is ${\rm curl}-\infty$ quasiconvex and condition \eqref{cahrQfinfty}.


\item 
The last assertion follows by the first part and $(7)$.

\noindent For the first part, start considering the functionals $F_p$ and $F$ introduced in \eqref{Ip} and \eqref{Fgammalimit}, as in the proof of (6). First, observe that \cite[(5) in Remark 5.2]{PZ}  guarantees that the $\Gamma$-limit with respect to the $L^\infty$ convergence of the restriction of the functionals $F_p$ to $ W^{1,\infty}(Q;\mathbb R^N)$ is given by the functional $F$ in \eqref{Fgammalimit}. Therefore, if we consider, for every $u \in W^{1,\infty}(Q;\mathbb R^N)$, the functional 
\begin{align}\label{Fcal}
	\mathcal F(u):=\inf\left\{
	\liminf_{p \to +\infty} \left(\int_Q f^p(D u_p)\,dx\right)^{\frac{1}{p}}:\ u_p \in W^{1,\infty}(Q;\mathbb R^N), \right.
	\\
	\left. \sup_ p\|u_p\|_{W^{1,\infty}} <+\infty, u_p \to u \hbox{ in }L^\infty  \right \}, \nonumber
\end{align}
one has
 \begin{align}\label{<=Fcal}
 	\operatorname*{ess\,sup}_{x \in  Q} \mathcal Q_\infty f(D u(x))=F(u) \leq \mathcal F(u) \hbox{ for every }u \in W^{1,\infty}(Q;\mathbb R^N).
 	\end{align}
 
In order to prove the opposite inequality, {\sl i.e.} 
\begin{align}\label{>=Fcal}
\mathcal F(u) \leq \operatorname*{ess\,sup}_{x \in  Q} \mathcal Q_\infty f(D u(x)) \hbox{ for every }u \in W^{1,\infty}(Q;\mathbb R^N),
\end{align} we start observing that, under our coercivity $(C)$ assumptions on $f$, by \cite[Theorem 9.1]{Dbook} and \cite[Proposition 6.16]{DM}, for every $p$, the functional
\begin{align}
\inf\left\{\liminf_{n \to +\infty}\left(\int_Q f^p(D u_n(x))\right)^{\frac{1}{p}}: \sup_n\|u_n\|_{W^{1,\infty}}<+\infty, u_n \to u \hbox{ in }L^\infty\right\} \nonumber\\
=\left(\int_Q Q(f^p)(D u(x))\,dx\right)^{\frac{1}{p}}.\label{relpinfty}
\end{align} 
Hence, by \cite[Proposition 6.11]{DM}, in view of \eqref{relpinfty}, \eqref{Fcal} can be written as
\begin{align*}
\mathcal F(u)= 
\inf\left\{
\liminf_{p \to +\infty} \left(\int_Q Q(f^p)(D u_p)\,dx\right)^{\frac{1}{p}}:\ u_p \in W^{1,\infty}(Q;\mathbb R^N), \right.
\\
\left. \sup_ p\|u_p\|_{W^{1,\infty}} <+\infty, u_p \to u \hbox{ in }L^\infty  \right \}.
\end{align*}

Thus,
 \begin{align*}
 \mathcal F(u) &\leq  \liminf_{p\to +\infty} \left(\int_Q Q(f^p)(D u(x))\,dx\right)^{\frac{1}{p}}\\
  &\leq \liminf_{p\to +\infty} \left(\operatorname*{ess\,sup}_{x \in Q} Q(f^p) (D u(x))\right)^{\frac{1}{p}} \\
  &=\liminf_{p\to +\infty}\operatorname*{ess\,sup}_{x \in Q} (Q(f^p))^{\frac{1}{p}}(D u(x))\\
  &\leq \operatorname*{ess\,sup}_{x \in Q}\mathcal Q_\infty f(D u(x))
 \end{align*}
where we have exploited \cite[Proposition 6.8 and 6.11]{DM}, and the definition of $\mathcal Q_{\infty} f$ and the monotonicity of $(Q(f^p))^{1/p}$.
 
Now we have that 
\begin{align}\label{Fcalweak}
\mathcal F(u)=\inf\left\{
\liminf_{p \to +\infty} \left(\int_Q f^p(D u_p(x))\,dx\right)^{1/p}: u_p \in W^{1,\infty}(Q;\mathbb R^N), \right.
\\
\left. \sup_ p\|u_p\|_{W^{1,\infty}} <+\infty, u_p \overset{\ast}{\rightharpoonup} u \hbox{ in }W^{1,\infty}(Q;\mathbb R^N)  \right \}, \nonumber
\end{align}
Indeed, the inequality  `$\leq $' is consequence of Rellich theorem and the opposite one follows by Banach-Alaoglu-Bourbaki theorem. 

To finish our proof, we recall that, under the assumption of $\rm curl$-Young quasiconvexity on $f$, \cite[ Theorem 4.1]{APESAIM} provides 
\begin{align}\label{<=weakFcal}
	\operatorname*{ess\,sup}_{x \in Q} f(D u(x)) \leq \liminf_{p \to +\infty} \left(\int_Q f^p(D u_p(x))\,dx \right)^{1/p},
\end{align}  for every $u_p \overset{\ast}{\rightharpoonup} u$ in $W^{1,\infty}(Q;\mathbb R^N)$, i.e.
$$\operatorname*{ess\,sup}_{x \in Q} f(D u(x)) \leq \mathcal F(u) \hbox{ for every }u \in W^{1,\infty}(Q;\mathbb R^N).$$
The opposite inequality is a consequence of the standard approximation of $L^\infty$ norm by the $L^p$ one, {\sl i.e.}
\begin{align*}
\mathcal F(u) \leq \liminf_{p\to +\infty}\left(\int_Q f^p(D u(x))\,dx\right)^{\frac{1}{p}}=\operatorname*{ess\,sup}_{x \in Q }f(D u(x)).
\end{align*}

Finally, putting together the representations for $\mathcal F$, obtained for \eqref{Fcal} and its equivalent \eqref{Fcalweak}, we can conclude that 
$$\operatorname*{ess\,sup}_{x \in  Q} f(D u(x)) = \operatorname*{ess\,sup}_{x \in  Q} \mathcal Q_\infty f(D u(x))$$
for every $u \in W^{1,\infty}(Q;\mathbb R^N)$. This, in turn, provides the equality $f(\xi)= \mathcal Q_\infty f(\xi)$ for every $\xi \in \mathbb R^{N \times n}$, which proves the $\rm curl-\infty$ quasiconvexity of $f$.

\item By (5), to prove the first statement it is enough to consider the case $f$ is ${\rm curl}$-Young quasiconvex. Then the result follows from \cite[Proposition 6.1]{APESAIM}, which in turn relies on the lower semicontinuity result in Theorem \ref{LsccurlYoung} (see also \cite[Theorem 4.2]{APESAIM} for another argument). 

For what concerns the second implication, let $\xi\in\mathbb{R}^{N\times n}$, $\varphi\in W_0^{1,\infty}(Q;\mathbb{R}^{N})$, $u(x)=\xi\cdot x+\varphi(x)$ and the push-forward measure $\nu=D u_{\sharp}\mathcal{L}^{N\times n}$. As observed in \cite[page 60]{AFP}, the measure $\nu$ coincides with the generalized product $\mathcal{L}^n\otimes\delta_{Du(x)}$. Considering \cite[Theorem 2.1]{KP1991}, we observe that this latter measure is a homogeneous gradient Young measure whose barycenter is $\xi$. Hence, by \eqref{fcurlYoungh},
\begin{align*}
	f(\xi)=f\left(\int_Q Du(x)\,dx\right)=f\left(\int_{\mathbb{R}^{N\times n}} \xi\,d\nu(\xi)\right)\leq \operatorname*{\nu-ess\,sup}_{\xi \in \mathbb{R}^{N\times n}}f(\xi)=\\
	=\operatorname*{\mathcal{L}^n\otimes\delta_{Du(x)}-ess\,sup}_{(x,\xi) \in Q\times\mathbb{R}^{N\times n}}f(\xi)= \operatorname*{ess\,sup}_{x \in  Q}f(Du(x))
\end{align*}
proving the weak Morrey quasiconvexity.

For the next implication, let $\xi,\eta\in\mathbb{R}^{N\times n}$ be such that $\mathrm{rank}(\xi-\eta)=1$ and let $0<\lambda<1$. One needs to show that 
$$f(\lambda\xi+(1-\lambda)\eta)\le \max\{f(\xi),f(\eta)\}.$$ Considering the construction of approximate solutions carried out in \cite[page 97]{mullerlecturenotes}, one gets a sequence $(u_j)\subseteq W^{1,\infty}(Q;\mathbb{R}^N)$ such that $u_j(x)=(\lambda\xi+(1-\lambda)\eta)\cdot x$ on $\partial Q$ and ${\rm dist}(Du_j,\{\xi,\eta\})\to 0$ in measure in $Q$. As discussed in \cite[page 120]{mullerlecturenotes}, $Du_j$ generates the homogeneous Young measure $\nu_x=\lambda\delta_\xi+(1-\lambda)\delta_\eta$. This measure has $\lambda\xi+(1-\lambda)\eta$ as barycenter, thus 
\begin{align*}
	f(\lambda\xi+(1-\lambda)\eta)=f\left(\int_{\mathbb{R}^{N\times n}}\zeta\,d\nu_x(\zeta)\right)\leq \operatorname*{\nu_x-ess\,sup}_{\zeta \in \mathbb{R}^{N\times n}} f(\zeta)=\max\{f(\xi),f(\eta)\}. 
\end{align*}
where we have also used \eqref{fcurlYoungh}.


\item
The first implication is a consequence of \cite[Proposition 5.2]{PZ} together with \eqref{cahrQfinfty}.

\noindent For what concerns the second implication since $f$ is strong Morrey quasiconvex and, invoking Proposition \ref{proposition independence of domain}, we get that $f$ is also strong Morrey quasiconvex in any other cube $C$. Hence, using (3) in Theorem \ref{thm relations between convexity notions}, we conclude that $f$ is periodic-weak Morrey quasiconvex in the cube $C$. The last parts are also consequences of (3) and (6) in Theorem \ref{thm relations between convexity notions}.
\end{enumerate}
\end{proof}

We observe that many implications in Theorem \ref{Other} do not invert. Indeed, we have the following result.

\begin{example}	\label{EXasymptoticresults}
\begin{enumerate}
\item In \cite[Example 6.6]{APESAIM} it has been shown that the function ${\arctan}({\rm det}): \xi \in \mathbb R^{2 \times 2} \to [-\pi/2,\pi/2]$ is polyquasiconvex (hence ${\rm curl}$-Young quasiconvex)  but
not level convex.
\item The example provided in \cite[Example 6.7]{APESAIM} shows that to ensure that a ${\rm curl}$-Young quasiconvex function is ${\rm curl}-\infty$ quasiconvex, some additional assumption, like $(C)$ in (8) of Theorem \ref{Other}, should be imposed. Indeed the function $W:\mathbb R^{2 \times 2} \to \mathbb [0,+\infty]$, defined as 
	$$W(\xi)= \sup\{h(|\xi|), k(\xi)\},$$ with $h$ and $k$ given by
	$k(\xi):= \arctan (\rm det \xi)$ and 
	$h(t)= \left\{
	\begin{array}{ll} 0 &\hbox{ if } t \leq 1,
		\\
		t-1 &\hbox{ if } 1 \leq  t \leq 2,\\
		1 &\hbox{ if  }t \geq 2\end{array}
	\right. $
	is ${\rm curl}$-Young quasiconvex but neither ${\rm curl}-\infty$ quasiconvex, nor quasiconvex,  nor level convex.
	
\item In \cite[Proposition 5.9]{APSIMA} it has been observed that 
the continuous function $f : \mathbb R \to [0,+\infty)$ given by
	$f(t) :=\left\{
	\begin{array}{ll}
		0  & \hbox{ if } t \leq 0,\\
	t & \hbox{ if } 0 \leq t\leq 1,\\
	1 &\hbox{ if } t \geq 1.
\end{array}
	\right.$
	is periodic-weak Morrey quasiconvex and strong Morrey quasiconvex since it is level
	convex. On the other hand, it has been proven that $f$ is not ${\rm curl}-\infty$ quasiconvex.
\end{enumerate}
\end{example}

\begin{proposition}

\begin{enumerate}
\item There exist ${\rm curl}$-Young quasiconvex functions which are not level convex, i.e. (1) in Proposition \ref{Other} does not invert.
\item There exist level convex and polyquasiconvex functions which are not ${\rm curl}-\infty$ quasiconvex. 
\item A function $f$ which satisfies \eqref{fcurlYoungh}  for every homogeneous gradient Young measure is not necessarily ${\rm curl}-\infty$ quasiconvex. Moreover, even if the function is assumed lower semicontinuous and bounded from below, \eqref{fcurlYoungh} does not imply that the function is ${\rm curl}-\infty$ quasiconvex, i.e. $\rm curl$-Young quasiconvexity does not imply ${\rm curl}-\infty$ quasiconvexity. 
\item There exist lower semicontinuous level convex (hence lower semicontinuous and polyquasiconvex, strong Morrey quasiconvex, periodic-weak Morrey quasiconvex, weak Morrey quasiconvex and rank-one quasiconvex)  functions which are not ${\rm curl}-\infty$ quasiconvex. 

\end{enumerate}
\end{proposition}
\begin{proof}
	Condition (1) follows from Example \ref{EXasymptoticresults}(1).
	
	Condition (2) relies on the fact that ${\rm curl}-\infty$ quasiconvex functions are lower semicontinuous (cf. Proposition \ref{curlinftyqcxlsc}), while level convexity and polyquasiconvexity do not entail lower semicontinuity (cf. Example \ref{exampleB}).

The first part of (3) follows by the same argument as in (2).
Indeed, condition \eqref{fcurlYoungh} follows from (1) in Theorem \ref{Other} and by Proposition \ref{curlinftyqcxlsc}, a function that is not lower semicontinuous, it is not ${\rm curl}-\infty$ quasiconvex. To prove the second assertion of (3), it suffices to refer to  Example \ref{EXasymptoticresults}(2).

The proof of (4) follows by  Example \ref{EXasymptoticresults}(3).
\end{proof}

In the following result we show a partial inversion of (3) in Proposition \ref{Other}. It stems ideas from \cite[Theorem 4.1]{BM} and \cite[Theorem 5.46]{Dbook}. 

\begin{proposition}\label{BallMuratProposition}
Let $f:\mathbb R^{N\times n}\to [0,+\infty)$ be a ${\rm curl}-\infty$ quasiconvex function, of the form 
\begin{align}
	\label{gaff}f(\xi)= g(\alpha + \langle \beta, T(\xi) \rangle),
\end{align}
with $\alpha \in \mathbb R$ and $\beta \in \mathbb R^{\tau(N, n)}$ (see \cite[ Theorem 5.46]{Dbook}),
or in particular  $f:\mathbb R^{N\times N}\to [0,+\infty)$
\begin{align}\label{gdet}
	f(\xi):=g (\det\xi),\end{align}
where $g:\mathbb R\to [0,+\infty)$ is a lower semicontinuous function and
there exist $a_1 >0, a_2 \in \mathbb R$ such that for every $\delta \in \mathbb R$
$$g(\delta)\geq a_1 |\delta|+ a_2.$$
Then $g$ is level convex, hence $f$ is polyquasiconvex.
\end{proposition}

\begin{proof}[Proof]
By \eqref{cahrQfinfty}, \cite[Theorem 6.24]{Dbook},\cite[Proposition 5.1]{PZ} and the arguments used in (2) of Proposition \ref{Other} to weaken the coercivity condition,
\begin{align}\label{cargdetinfty}
f(\xi) = \lim_{p\to +\infty} \left(\mathcal{Q}(f^p)(\xi)\right)^{\frac{1}{p}} 
 = \lim_{p\to +\infty} (\mathcal{Q}(g^p(\alpha + \langle \beta, T(\xi)\rangle)))^\frac{1}{p} \\
 = \lim_{p\to +\infty} ((g^p)^{\ast \ast}(\alpha + \langle \beta, T(\xi)\rangle)))^{\frac{1}{p}}
= g^{\rm lslc}(\alpha + \langle \beta, T(\xi)\rangle)),
\end{align}
where $g^{\rm lslc}$ denotes the lower semicontinuous and level convex envelope of $g$ (see \cite{RZ}). Hence $f$ is polyquasiconvex.
Putting together \eqref{gdet} and \eqref{cargdetinfty}, from the arbitrariness of $\alpha + \langle \beta, T(\xi)\rangle \in \mathbb R$, we obtain the level convexity of $g$.
\end{proof}

The previous analysis leaves several open questions about the convexity notions treated in this section. 

\begin{enumerate}
\item We have seen that ${\rm curl}_{(p>1)}$-Young quasiconvexity, ${\rm curl}-$Young quasiconvexity, ${\rm curl}-\infty$ quasiconvexity, imply strong Morrey quasiconvexity. 
One may wonder if strong Morrey quasiconvexity is indeed below in this hierarchy, in the class of coercive functions. Note that the function in Example \ref{EXasymptoticresults} (3) is not coercive. In the same spirit one may wonder if weak Morrey quasiconvexity entails ${\rm curl}-\infty$ quasiconvexity in the class of functions satisfying $(C)$.
\item In (7) and (8) we have proven that  ${\rm curl}_{(p>1)}$-Young quasiconvexity, ${\rm curl}-$Young quasiconvexity and ${\rm curl}-\infty$ quasiconvexity are equivalent, but it is not known if the assumptions are sharp.

\item In (8) we obtained that ${\rm curl}$-Young quasiconvex functions are ${\rm curl}_{(p>1)}$-Young quasiconvex requiring that $f$ is non-negative since our proof relies on proving that $f$ is also ${\rm curl}-\infty$ quasiconvex.
One may wonder if a direct proof can be provided removing this artificial assumption.

\end{enumerate}

 \begin{remark}\label{finalrem4}
 As regard as question (1) dealing with coercivity, it is worth to recall that \cite[Theorem 4.2]{APSIMA} in the $\rm curl$-free setting, has been improved in \cite[Theorem 2.2]{PZ}. Indeed, the Carath\'eodory assumption has been relaxed to Lebesgue $\otimes$ Borel-measurability for $f$, but under a linear growth constraint from above and below on $f(x, \cdot)$ for a.e. $x \in \Omega$, obtaining a strong Morrey quasiconvex limiting density, $f_\infty$ (see \cite[eq. (61) and (62)]{PZ}), i.e. $\Gamma-\lim_{p\to +\infty}(\int_\Omega f^p(x, \nabla u(x))dx)^\frac{1}{p} = \operatorname*{ess\,sup}_{x \in \Omega} f_\infty(x,\nabla u(x))$. Moreover, it has been proven that $f_\infty$ coincides with $Q_\infty f$ either assuming in addition upper semicontinuity on $f(x,\cdot$) or in the homogeneous setting, requiring just Borel measurability and (up to a constant) \eqref{coerci} (see \cite[Remark 5.2]{PZ}). This latter case reinforces the curiosity towards proving (or disproving) that strong Morrey quasiconvexity and coercivity could imply ${\rm curl}-\infty$ quasiconvexity.
 
 \smallskip
 Moreover by the above mentioned $\Gamma$-convergence results in \cite{APSIMA} and \cite{PZ}, it follows that ${\rm curl}-\infty$ quasiconvexity is necessary and sufficient for the $L^p$ approximation under continuity and coercivity assumptions, but it is not completely clear if this notion is really necessary under the sole assumption of lower semicontinuity. 
  
 Then, question (1) above can be rephrased as follows: despite the scalar case, as already emphasized in Section 3, it is currently open the relaxation for supremal functionals, leaving ${\rm curl}-\infty$ quasiconvexity (and ${\rm curl}$-Young quasiconvexity) just a sufficient condition for the lower semicontinuity of functionals of the type \eqref{suppl}, but it is not known its necessity even under coercivity assumptions.  In the same order of ideas, departing from the results in \cite{PZ} one can observe that strong-Morrey quasiconvexity is necessary for power-law approximations, while it is currently open its sufficiency.

 \smallskip
 
 Finally, regarding the scalar case, we underline that, in the homogeneous and scalar case in \cite[Remark 5.2]{PZ}, i.e. $v(x)=\nabla u(x)\in \mathbb R^{n\times N}$, with either $n=1$ or $N=1$, the limiting density $f_\infty$ coincides with the level convex and lower semicontinuous envelope of $f$, thus recovering the results in \cite{APESAIM}, cf. Proposition 2.9 and Theorem 2.10 therein, and the particular cases discussed in \cite{GNP, BGP, CDP, BPZ}, and leads by means of power-law approximation strategies, to the same conclusions as $(viii)$ in Remark \ref{Counteremark}.
 \end{remark}


\color{black}
\section{Relating convexity definitions from the integral and supremal settings}

It is easily seen that convexity is a sufficient condition to level convexity. Indeed, we will see that the convexity notions arising on the  minimization of integral functionals are strictly stronger than the corresponding ones  
appearing in the minimization of supremal functionals.

\begin{proposition}\label{4.13} Let $f :\mathbb R^{N\times n}\to \mathbb R$.
\begin{enumerate}
\item\label{pa} If $f$ is polyconvex, {\sl i.e.}, for some convex function $g:\mathbb{R}^{\tau(n,N)}\longrightarrow \mathbb{R}$, 
$$f(\xi)=g(T(\xi)), \forall\ \xi\in\mathbb{R}^{N\times n}$$ being $\tau(n,N)$ and $T(\xi)$ as in Definition \ref{def convex notions}, then $f$ is polyquasiconvex.
\item\label{3a}  If $f$ is quasiconvex, then $f$ is ${\rm curl}$-Young quasiconvex.
\item\label{qia} If $f$ is quasiconvex, then $f$ is strong Morrey quasiconvex. 
\item\label{qa} If $f$ is quasiconvex and non-negative then $f$ is ${\rm curl}-\infty$ quasiconvex.
\item \label{rkone} If $f$ is rank-one convex, {\sl i.e.} $f$ is convex along rank-one directions, then $f$ is rank-one quasiconvex. 
\end{enumerate}
\end{proposition}

\begin{remark}\label{remprop4.13}
The proof of \eqref{qa} is presented for the reader's convenience since it is given by \cite[Proposition 3.6]{APSIMA} for the $\mathcal A$-free setting, under the extra explicit assumptions of local boundedness and upper semicontinuity. These requirements are not needed in the ${\rm curl}$-free setting, since they are granted for any quasiconvex function, see (iii) in Remark \ref{rem-qcxdef}.
\end{remark}

\begin{proof}
\eqref{pa} and \eqref{rkone} The results follow immediately from the fact that convex functions are also level convex.

\eqref{3a} The proof is in \cite[Proposition 3.3 (1)]{APESAIM}, in the more general case of the operator ${\rm curl}$ replaced by a generic partial differential operator $\mathcal A$ with constant rank and under the extra assumption that $f$ is continuous. By (iii) in Remark \ref{rem-qcxdef}, the continuity is implicitly guaranteed by the quasiconvexity of $f$.

\eqref{qia} Having in mind Proposition \ref{new characterization of quasiconvexity} it is enough to observe that a function satisfying \eqref{AltQ} is strong Morrey quasiconvex. To obtain this, it suffices to estimate the integral by the essential supremum.

\eqref{qa} It follows by the equalities
\begin{align*}
	f(\xi)=f^p(\xi)^{1/p}=\left(\inf\left\{\int_Q f^p(\xi+ D \varphi(y))dy:\varphi \in W^{1,\infty}_0(Q;\mathbb R^N)\right\}	\right)^{1/p}
\end{align*}
where in the second one it has been exploited the quasiconvexity of $f^p$ due to the quasiconvexity of $f$, which follows from the increasing monotonicity of $(\cdot)^p$, the Jensen's inequality applied to this latter function, and the fact that $\inf$ and $\frac{1}{p}$ power interchanges.
Hence it suffices to pass to the limit as $p \to +\infty$ to get the assertion.
\end{proof}

In view of the previous proposition, one might wonder if the supremal convexity notions imply some integral convexity notions. This is not likely to happen since clearly level convexity does not imply rank-one convexity.
The next example taken from \cite[Example 5.5]{APSIMA} proves our claim considering a continuous, non-negative function with linear growth defined in $\mathbb R^{N\times n}$. 

\begin{example}	\label{EXAMPLE-D}
Let $f : \mathbb R^{N\times n} \to [0,+\infty)$ be the continuous function given by
$$f(\xi) :=\left\{
\begin{array}{ll} 
|\xi| &\hbox{ if } |\xi|\leq 1,\\
1 &\hbox{ if } 1\leq |\xi|\leq 2,\\
\frac{1}{2}|\xi| &\hbox{ if } |\xi| \geq 2.
\end{array}
\right.$$
Then $f$ 
is level convex (see (2) in Theorem \ref{Other}) but it	is 
not rank-one convex.
\end{example}

\begin{proposition}\label{cexint} Let $N, n \in \mathbb N$.
There exist functions which are level convex (and thus polyquasiconvex, ${\rm curl}-\infty$ quasiconvex, ${\rm curl}_{(p>1)}$ quasiconvex, ${\rm curl}$-Young quasiconvex, strong Morrey quasiconvex, rank-one quasiconvex) which are not rank-one convex (hence neither convex, nor polyconvex, nor quasiconvex).
\end{proposition}
\begin{proof}[Proof] 
The result follows from Example \ref{EXAMPLE-D}, Proposition \ref{Other} and Theorem \ref{thm relations between convexity notions}, and \cite[Theorem 5.3]{Dbook}.

\end{proof}


\appendix\section{}

In \cite{BJW} it was proven that a necessary and sufficient condition to the sequential weak* lower semicontinuity in $W^{1,\infty}(\Omega,\mathbb{R}^N)$ of a functional
$$F(u):=\operatorname*{ess\,sup}_{x\in \Omega}f\left(Du\left( x\right)  \right),\quad u\in W^{1,\infty}(\Omega,\mathbb{R}^N)$$
is the strong Morrey quasiconvexity of the supremand $f$. Here, we exploit the arguments employed to achieve this statement, with the ultimate goal of showing that the cube $Q$ in the definition of strong Morrey quasiconvexity can be replaced by other sets in an appropriate class, cf. Proposition \ref{proposition independence of domain}.  

\begin{proposition}\label{Strong Morrey implies seq weak* lsc}
	Let $\Omega\subseteq\mathbb{R}^{n}$ be a bounded, convex, and open set, and let $f:\mathbb{R}^{N\times n}\longrightarrow \mathbb{R}$ be a strong Morrey quasiconvex on the set $\Omega$. That is, $f$ satisfies \eqref{SMQcxOmega}.
	Let $O\subseteq\mathbb{R}^{n}$ be a bounded open set. Then, the functional $F(u,O):=\operatorname*{ess\,sup}_{x\in O}f\left(Du\left( x\right)  \right)$ is sequentially weakly* l.s.c. in $W^{1,\infty}(O,\mathbb{R}^N)$.
\end{proposition}

\begin{proof}
The result follows as in \cite[Theorem 2.6]{BJW}. By Remark \ref{translfSMQcx}, we can reduce to the case that $\Omega$ contains the origin. Then invoking the strong version of Besicovitch derivation theorem (cf. \cite[Theorem 5.52]{AFP}) that allows to work with more general sets than the cube, that need to be convex. Moreover we make use of Lemma \ref{2.5inOmegabddopen0} instead of \cite[Proposition 2.5]{BJW}. Note that bounded, open, convex sets have Lipschitz boundary thus the hypotheses of the lemma are fulfilled.
\end{proof}

The next goal is to show that strong Morrey quasiconvexity is also a necessary condition to sequential weak* lower semicontinuity of $F(\cdot,O)$ in $W^{1,\infty}(O;\mathbb{R}^N)$. Again, we follow the same procedure of the proof of \cite[Theorem 2.7, Lemma 2.8.]{BJW}. We stress that we don't require the sequential weak* lower semicontinuity in any set as in \cite{BJW} because we work in the same set, either in the sequential weak* lower semicontinuity of the functional and in the strong Morrey quasiconvexity. This is not a restriction if we combine it with Proposition \ref{Strong Morrey implies seq weak* lsc}, as it will be made clear in Proposition\ref{proposition independence of domain}. 

We start considering weak Morrey quasiconvexity.

\begin{proposition}\label{seq weak* lsc implies weak Morrey}
	Let $f:\mathbb{R}^{N\times n}\longrightarrow \mathbb{R}$ be a Borel measurable function and consider the functional $$F(u,\Omega):=\operatorname*{ess\,sup}_{x\in \Omega}f\left(Du\left( x\right)  \right),$$ where $\Omega$ is a bounded open set of $\mathbb{R}^n$. If $F(\cdot,\Omega)$ is sequentially weakly* l.s.c. in $W^{1,\infty}(\Omega;\mathbb{R}^N)$ then $f$ is lower semicontinuous and it is periodic-weak Morrey quasiconvex (considering any cube in place of $Q$). In particular, $f$ is also weak Morrey quasiconvex (considering also any cube in place of $Q$). 
\end{proposition}

\begin{remark}
This result together with the Proposition \ref{lemma weak independent} entail that, if $F$ is sequentially weakly* l.s.c. in $W^{1,\infty}(\Omega;\mathbb{R}^N)$, then $f$ is weak Morrey quasiconvex in any bounded and open set $\Omega$ with boundary of null $\mathcal{L}^n$-measure.
\end{remark}

\begin{proof} 
The proof follows as in \cite[Lemma 2.8]{BJW}. The lower semicontinuity of $f$ follows simply considering, for a given sequence $(\xi_k)_{k\in\mathbb{N}}\subseteq \mathbb{R}^{N\times n}$ converging to $\xi\in\mathbb{R}^{N\times n}$, the sequence of affine functions $u_k(x)=\xi_k\, x$ that converges to $u(x):=\xi\, x$ in $W^{1,\infty}(\Omega;\mathbb{R}^N)$ since $\Omega$ is bounded.
	
	It remains to prove that $f$ is periodic-weak Morrey quasiconvex since, in view of Theorem \ref{thm relations between convexity notions} \eqref{(4)}, the weak Morrey quasiconvexity then follows immediately. Let $C$ be any cube in $\mathbb{R}^n$, $\xi\in\mathbb{R}^{N\times n}$, and $\varphi\in W_{\rm{per}}^{1,\infty}(C;\mathbb{R}^N).$ We want to show that
	$$\displaystyle f(\xi)\le \operatorname*{ess\,sup}_{x\in C}f\left( \xi+ D\varphi\left( x\right)  \right).$$ 
	
	Then consider $\varphi_\varepsilon(x):=\varepsilon \varphi\left(\frac{x}{\varepsilon}\right)$. Of course, as $\varepsilon$ tends to 0, $\varphi_\varepsilon$ converges weakly* in  $W^{1,\infty}(\Omega;\mathbb{R}^N)$ to 0. Therefore $u_\varepsilon(x):=\xi x+ \varphi_\varepsilon(x)$ converges weakly* in  $W^{1,\infty}(\Omega;\mathbb{R}^N)$ to $u_\xi(x):=\xi\,x$, as $\varepsilon$ tends to 0. Applying the sequential weak* l.s.c. of $F(\cdot,\Omega)$, one gets
	$$\begin{array}{l}\displaystyle{f(\xi)=\operatorname*{ess\,sup}_{x\in \Omega}f\left( Du_\xi\left( x\right)  \right)\le \liminf_{\varepsilon\to 0}\operatorname*{ess\,sup}_{x\in \Omega}f\left( Du_\varepsilon \left( x\right)  \right)=}\vspace{0.2cm}\\
		
		\displaystyle{=\liminf_{\varepsilon\to 0}\operatorname*{ess\,sup}_{x\in \Omega}f\left(\xi+ D \varphi \left( \frac{x}{\varepsilon}\right)\right)\le \operatorname*{ess\,sup}_{z\in C}f\left( \xi+D\varphi \left(z\right)\right)}
	\end{array}$$ 
	where we used in the last inequality the periodicity of $\varphi$.
\end{proof}

We can now consider the strong Morrey quasiconvexity.

\begin{proposition}\label{seq weak* lsc implies strong Morrey}
	Let $f:\mathbb{R}^{N\times n}\longrightarrow \mathbb{R}$ be a Borel measurable function and consider the functional $$F(u,\Omega):=\operatorname*{ess\,sup}_{x\in \Omega}f\left(Du\left( x\right) \right),$$ where $\Omega$ is a bounded open set of $\mathbb{R}^n$ with Lipschitz  boundary
	. If $F(\cdot,\Omega)$ is sequentially weakly* l.s.c. in $W^{1,\infty}(\Omega;\mathbb{R}^N)$ then $f$ is strong Morrey quasiconvex in $\Omega$.\end{proposition}

\begin{proof} 
	The proof is the same as the one of \cite[Theorem 2.7]{BJW}. That is, assume by contradiction that there are $\varepsilon>0$, $\xi\in \mathbb{R}^{N\times n}$, and $K>0$ such that, for every $\delta>0$ there is a function 
	$$\varphi_\delta\in W^{1,\infty}(\Omega;\mathbb{R}^N)$$ 
	such that
	$$||D\varphi_\delta||_{L^\infty(\Omega)}\le K,\quad \max_{x\in\partial \Omega}|\varphi_\delta(x)|\le \delta\quad\text{and}\quad f(\xi)> \operatorname*{ess\,sup}_{x\in \Omega} f(\xi+D\varphi_\delta(x))+\varepsilon.$$
	It can be ensured that, up to a subsequence, $\varphi_\delta$ weakly* converges in $W^{1,\infty}(\Omega;\mathbb{R}^N)$ to some function $\varphi \in W^{1,\infty}_0(\Omega;\mathbb{R}^N)$. Indeed, $\Omega$ being a bounded open set with Lipschitz boundary, by Sobolev's embedding theorems, 
	 $\varphi_\delta \in C(\overline{\Omega};\mathbb R^N)$. Moreover, 
	 the functions $\varphi_\delta$ are Lipschitz continuous with Lipschitz constant bounded by $||D\varphi_\delta||_{L^\infty(\Omega)}\le K$. Therefore, the sequence $\varphi_\delta$ is equicontinuous and by Ascoli-Arzel\'a's theorem, up to a subsequence, $\varphi_\delta$ converges uniformly to a function $\varphi \in C(\overline{\Omega};\mathbb R^N)$. On the other hand, the Lipschitz inequality enjoyed by the functions $\varphi_\delta$, together with $\max_{x\in\partial \Omega}|\varphi_\delta(x)|\le \delta$, ensure that $\varphi_\delta$ is a bounded sequence in $W^{1,\infty}(\Omega;\mathbb R^N)$. Thus, up to a subsequence, it converges weakly* in $W^{1,\infty}(\Omega;\mathbb R^N)$. We can then conclude that the uniform limit, $\varphi$, belongs to $W^{1,\infty}(\Omega;\mathbb R^N)$ and it is also the weak* limit in $W^{1,\infty}(\Omega;\mathbb R^N)$. Finally the estimates on the traces on $\partial \Omega$ of $\varphi_\delta$, also imlpy that $\varphi \in W^{1,\infty}_0(\Omega;\mathbb R^N)$.
	
	Then, by the sequential weak* l.s.c. of $F(\cdot,\Omega)$, one gets
	$$\operatorname*{ess\,sup}_{x\in \Omega}f\left(\xi+ D\varphi\left( x\right)  \right)\le	\liminf_{\delta\to 0} \operatorname*{ess\,sup}_{x\in \Omega}f\left(\xi+D\varphi_\delta \left( x\right)  \right)<f(\xi)-\varepsilon<f(\xi).$$ This contradicts the fact that, by Propositions \ref{seq weak* lsc implies weak Morrey} and \ref{lemma weak independent}, 
	$$\displaystyle f(\xi)\le \operatorname*{ess\,sup}_{x\in \Omega}f\left( \xi+ D\psi\left( x\right)  \right),\ \forall\ \xi\in\mathbb{R}^{N\times n},\ \forall\ \psi\in W_0^{1,\infty}(\Omega;\mathbb{R}^N),$$ where we used the fact that we are assuming $\Omega$ bounded, open, and with  Lipschitz, 
	thus with boundary of null measure.
\end{proof}

\section{}
The following result is a key tool to provide an alternative argument to the proof of \cite[Theorem 4.2]{APESAIM}. 
It extends to the inhomogeneous setting \cite[Lemma 2.5]{KZ}, where the continuity assumption of \cite[Theorem 34]{B1} was relaxed.

\begin{lemma}\label{Barronmonx}
	Let $U\subset \mathbb R^n$ be an open set with finite measure and let $f:U\times \R^m\to \R$  be a normal integrand bounded from below. 		Further, let $(u_k)$ be a uniformly bounded sequence of functions in $L^\infty(U;\mathbb R^m)$ generating a Young measure $\nu=\{\nu_x\}_{x\in U}$. Then,
	\begin{equation*}
		\liminf_{k \to \infty}\operatorname*{ess\,sup}_{x\in U} f(x,u_k(x))  \geq \operatorname*{ess\,sup}_{x\in U}\bar f(x),
	\end{equation*}
	where $\bar f(x) := \operatorname*{\nu_x-ess\,sup}_{\xi\in \R^m} f(x,\xi)$ for $x\in U$.
\end{lemma}

\begin{proof} 
	We give the details of the proof for the reader's convenience, besides it follows along the lines of \cite[Lemma 2.5]{KZ}. 
	
	Without loss of generality we can assume that $f$ is non negative. Let $\varepsilon>0$ be fixed, and choose a set $S\subset U$ with positive Lebesgue measure such that $\bar{f}(x)\geq \|\bar{f}\|_{L^\infty(U)}-2\varepsilon$ for all $x\in S$.
	Next, we show that there exists a measurable subset $S'\subset S$ with $\mathcal L^n(S')>0$ such that
	\begin{align}\label{est1}
		\Bigl(\int_{\R^m} |f(x,\xi)|^p \,d\nu_x(\xi)\Bigr)^{\frac{1}{p}} \geq \|\bar f\|_{L^\infty(U)}-\varepsilon
	\end{align}
	for all $x\in S'$ and $p>1$ sufficiently large. 
	Indeed, with 
	$$S_k := \Bigl\{x \in S: 
	\bigl(\textstyle \int_{\mathbb R^m}f(x,\xi)^p d\nu_x(\xi)\bigr)^{\frac{1}{p}}
	\geq  \| \bar f\|_{L^\infty(\Omega)} - \varepsilon \text{ for all $p \geq k$}\Bigr\}$$
	for $k \in \mathbb N$, one has that $S= \bigcup_{k=1}^\infty S_k$.
	Since $\mathcal L^n(S) > 0$, there must be at least one $k'$ for which $\mathcal L^d(S_{k'}) > 0$, and setting $S' := S_{k'}$
	shows~\eqref{est1}. 
	
	We take the inequality in~\eqref{est1} to the $p$th power and integrate over $S'$. Along with Theorem \ref{fundamental-theorem} (i), extended to normal integrands by using \cite[Corollary 6.30]{Fonseca-Leoni-Lp},  
	it follows that
	\begin{align*}
		\mathcal L^d(S') (\|\bar{f}\|_{L^\infty(\Omega)}-\varepsilon)^p & \leq \int_{S'} \int_{\R^m} |f(x, \xi)|^p\, d\nu_x(\xi) \, dx 
		\\ &\leq \liminf_{k\to \infty} \int_U |f(x, u_k(x))|^p \, dx \leq \liminf_{k\to \infty} \|f(\cdot, u_k(\cdot))\|_{L^\infty(U)}^p \mathcal L^d (U).
	\end{align*}
	Hence, 
	\begin{align*}
		\liminf_{k\to \infty} \|f(\cdot, u_k(\cdot))\|_{L^\infty(U)} \geq  \Bigl(\frac{\mathcal L^d(S')}{\mathcal L^d (U)}\Bigr)^{\frac{1}{p}}\bigl(\|\bar{f}\|_{L^\infty(U)}-\varepsilon\bigr)
	\end{align*}
	for $p>1$ sufficiently large. Letting $p\to \infty$ and recalling that $\varepsilon>0$ is arbitrary concludes the proof.
\end{proof}
We are in position to prove that $\rm curl$-Young quasiconvexity is a sufficient condition for the lower semicontinuity of supremal functionals. The result can be trivially extended to the $\mathcal A$-free setting, thus providing an alternative argument to the one proposed in \cite[Theorem 4.2]{APESAIM}.

\begin{theorem}
	\label{LsccurlYoung}
	Let $f:\Omega \times \mathbb R^{N\times n} \to \mathbb R$ be a normal integrand, bounded from below and such that  $f(x,\cdot)$ is ${\rm curl}-$ Young quasiconvex for a.e. $x \in \Omega$. Let $F:W^{1,\infty}(\Omega;\mathbb R^N)\to\mathbb R$ be the functional defined by 
	$$
	F(u)= \esssup_{x \in \Omega}f(x,D u(x)).
	$$
	Then the functional $F$ is sequentially weakly * lower semicontinuous in $W^{1,\infty}(\Omega;\mathbb R^N)$.
\end{theorem}

\begin{proof}
	The result follows from Lemma \ref{Barronmonx} and Definition \ref{def more convex notions} (2).
	Without loss of generality we can assume that $f$ is non negative.
	
	Indeed,	taken any sequence $(u_k)$ weakly * converging to $u$ in $W^{1,\infty}(\Omega;\mathbb R^N)$, which generates a $W^{1,\infty}$-gradient Young measure, $\nu_x$ with baricenter $D u(x)$, we have 
	\begin{align*}
		\displaystyle{
			\liminf_{k\to \infty}F(u_k)\geq \esssup_{x \in \Omega}\nu_x-\esssup_{\xi \in\mathbb R^{N\times n}} f(x,\xi)\geq f\left(x,\int_{\mathbb R^{N\times n}} \xi d \nu_x(\xi)\right) \hbox{ for a.e.} x \in \Omega}.\end{align*}
	Consequently
	\begin{align*}
		\liminf_{k\to \infty} F(u_k)
		\geq \esssup_{x \in \Omega }f(x,D u(x))= F(u),
	\end{align*}
	which proves the desired lower semicontinuity result.
\end{proof}

\section{}

The result below relates the level convexity of a function with a generalization of Jensen's inequality for the supremal setting. A proof can be found in Barron \cite[Theorem 30]{B1}, (see \cite[Theorem 1.2]{BJW}, where the theorem is stated under a lower semicontinuity hypothesis, and \cite[Lemma 2.4]{BJL} where one implication has been shown in order to provide a Hopf-Lax formula). It is possible to avoid this condition as already mentioned without proof in \cite{RZ}. For convenience of the reader we include the proof here.

\begin{theorem}\label{Jensensupremalscalar}
A Borel measurable function $f:\mathbb R^{n}\to \mathbb R$ is level convex if and only if it verifies the supremal Jensen's inequality:
\begin{equation}\nonumber
f\left(\int_\Omega \varphi\,d\mu\right)\leq \mu-\operatorname*{ess\,sup}_{x\in\Omega} f(\varphi(x))
\end{equation}
for every probability measure $\mu$ on $\mathbb{R}^{d}$ supported on the open set $\Omega \subseteq \mathbb{R}^{d}$, and every $\varphi \in L^1_\mu(\Omega;\mathbb{R}^{n})$.

In particular, considering the Lebesgue measure, if $\Omega$ is a set with finite Lebesgue measure, $$f\left(\frac{1}{|\Omega|}\int_\Omega\varphi(x)\,dx\right)\le \operatorname*{ess\,sup}_{x\in\Omega}f(\varphi(x)),\ \forall\ \varphi\in L^1(\Omega;\mathbb{R}^n).$$
\end{theorem}

\begin{proof} To show that a function satisfying the supremal Jensen's inequality is level convex, it suffices to take  for every $t \in [0,1]$ a function $\varphi$ whose values are $\xi$ on a set of $\mu$-measure $t$ and $\eta$ on a set of $\mu$-measure $1-t$.

Now we prove that, if $f$ is level convex then it satisfies supremal Jensen's inequality. To this end it is enough to prove that, given an open set $\Omega \subseteq \mathbb{R}^{d}$, for every probability measure $\mu$ on $\mathbb{R}^{d}$ supported on $\Omega$, if $\varphi \in L^1_\mu(\Omega;\mathbb{R}^{n})$ is such that $\varphi(x) \in C$, $\mu$-a.e. $x\in \Omega$ with $C\subset\mathbb{R}^{n}$ a convex set with finite dimension, then
\begin{equation}\label{MC}
\int_\Omega \varphi\, d \mu \in C.
\end{equation}
Indeed, once this is proved we obtain supremal Jensen's inequality taking $C$ as the level set of $f$ corresponding to $\displaystyle\mu-\operatorname*{ess\,sup}_{x\in\Omega} f(\varphi(x))$,  i.e. $\displaystyle \mu-L_{\operatorname*{ess\,sup}_{x\in\Omega} f(\varphi (x))}(f)$, which is convex under the level convexity assumption on $f$.

Thus, it remains to prove that under the above assumptions \eqref{MC} holds. We argue by an inductive argument on the dimension $N$ of the convex set $C$. Clearly there is nothing to prove if $N=0$. Assume  that $C$ is a convex set of dimension $N>0$, that is the smallest affine space containing $C$ has dimension $N$. If $\varphi(x) \in C$ for $\mu$-a.e. $x \in \Omega$, clearly $$\xi_0:=\int_\Omega \varphi \,d \mu \in \overline{C},$$ which is still convex. The thesis is obvious if $\xi_0$ belongs to the interior of $C$ relative to the smallest affine subset containing $C$, in symbols $\xi_0\in {\rm rel\,int}\,C$, thus we can reduce ourselves to the case when $\xi_0  \in \overline{C}\setminus {\rm rel\,int}\,C$, the relative boundary of $C$. By  \cite[Theorem 6.3]{Rockafellar}, $\xi_0  \in \overline{C}\setminus {\rm rel\,int}\,\overline{C}$ and thus, \cite[Theorem 11.6]{Rockafellar} ensures the existence of a non trivial supporting hyperplane $H=\{\xi\in\mathbb{R}^{n}:\ \langle \alpha,\xi\rangle=c\}$ to $\overline{C}$ containing $\xi_0$. That is, $\langle\alpha,\xi_0\rangle=c$ and $\langle\alpha,\xi\rangle\ge c,\ \forall\ \xi\in \overline{C}$. Therefore
$$\langle \alpha,\xi-\xi_0 \rangle\ge 0,\ \forall\ \xi\in \overline{C}$$ and taking, in particular, $\xi=\varphi(x)$, we get $$\langle \alpha,\varphi(x)-\xi_0\rangle\ge 0,\ \mu-a.e.\ x\in \Omega.$$ Integrating in $\Omega$, it results  $$\int_\Omega\langle \alpha,\varphi(x)-\xi_0\rangle\,d\mu(x)=\langle \alpha,\int_\Omega\varphi(x)-\xi_0\,d\mu(x)\rangle=0.$$

Therefore, $\langle \alpha,\varphi(x)-\xi_0\rangle= 0,\ \mu-a.e.\ x\in \Omega$ which means that $\mu-a.e.\ x \in \Omega$, $\varphi(x)\in H \cap C$ which is a convex set with dimension less than $N$, thus by induction hypothesis $\xi_0 \in H\cap C$, and thus it is also in $C$. That concludes the proof.
\end{proof}

\color{black}




\section*{Acknowledgments}
This work is funded by national funds through the FCT Fundacao para a Ciencia e a Tecnologia, I.P., under the scope of the projects UIDB/00297/2020 and UIDP/00297/2020 (Center for Mathematics and Applications).

AMR aknowledges the support of INdAM-GNAMPA through Progetto Professori Visitatori 2021 and the hospitality of Dipartimento di Scienze di Base ed Applicate per l'Ingegneria di Sapienza - University of Rome.

The last author is a member of INdAM-GNAMPA whose support, through the  projects 
``Prospettive nella scienza dei materiali: modelli variazionali, analisi
asintotica e omogeneizzazione'' (2023),  ``Analisi variazionale di modelli non-locali
nelle scienze applicate'' (2020)  is gratefully ackowledged.
EZ is also indebted with  Departamento de Matem\'{a}tica, Faculdade de Ci\^{e}ncias e Tecnologia, Universidade Nova de Lisboa for its kind hospitality.

\end{document}